\documentclass[12pt,reqno,a4paper]{amsart}
\usepackage{graphicx}
\usepackage{amssymb,amscd,textcomp,mathrsfs}
\usepackage{hyperref}

\DeclareSymbolFont{yhlargesymbols}{OMX}{yhex}{m}{n} 
\DeclareMathAccent{\yhwidehat}{\mathord}{yhlargesymbols}{"62}

\begin{document}

\let\kappa=\varkappa
\let\eps=\varepsilon
\let\phi=\varphi

\def\Z{\mathbb Z}
\def\R{\mathbb R}
\def\C{\mathbb C}
\def\Q{\mathbb Q}

\def\OO{\mathcal O}
\def\AA{\mathcal A}

\def\RR{\mathscr R}
\def\PP{\mathscr P}
\def\MM{\mathscr M}

\def\CP{\C{\mathrm P}}
\def\RP{\R{\mathrm P}}
\def\conj{\overline}
\def\Beta{\mathrm{B}}
\def\p{\partial}
\def\wh{\widehat}
\def\wt{\widetilde}

\renewcommand{\Im}{\mathop{\mathrm{Im}}\nolimits}
\renewcommand{\Re}{\mathop{\mathrm{Re}}\nolimits}

\newcommand{\codim}{\mathop{\mathrm{codim}}\nolimits}
\newcommand{\id}{\mathop{\mathrm{id}}\nolimits}
\newcommand{\Aut}{\mathop{\mathrm{Aut}}\nolimits}
\newcommand{\supp}{\mathop{\mathrm{supp}}\nolimits}
\newcommand{\dist}{\mathop{\mathrm{dist}}\nolimits}

\newcommand{\rh}[1]{\mathop{\RR\text{-}\mathrm{hull}}\nolimits(#1)}
\newcommand{\ph}[1]{\mathop{P(\conj{\Omega})\text{-}\mathrm{hull}}\nolimits(#1)}
\newcommand{\oh}[1]{\mathop{\OO\text{-}\mathrm{hull}}\nolimits(#1)}
\newcommand{\mh}[1]{\mathop{\MM\text{-}\mathrm{hull}}\nolimits(#1)}
\newcommand{\phbig}[1]{\mathop{P(\conj{\Omega})\text{-}\mathrm{hull}}\nolimits\bigl( #1 \bigr)}

\newtheorem{mainthm}{Theorem}
\renewcommand{\themainthm}{{\Alph{mainthm}}}

\newtheorem{thm}{Theorem}[section]
\newtheorem{lem}[thm]{Lemma}
\newtheorem{prop}[thm]{Proposition}
\newtheorem{cor}[thm]{Corollary}

\theoremstyle{definition}
\newtheorem{df}[thm]{Definition}
\newtheorem{rem}[thm]{Remark}
\newtheorem{exm}[thm]{Example}

\numberwithin{equation}{section}

\title{Hulls and boundaries}

\author[Nemirovski]{Stefan Nemirovski}
\address{%
Bergische Universit\"at Wuppertal, Germany;\newline 
\strut\hspace{9 true pt} Steklov Mathematical Institute, Moscow, Russia}
\email{stefan.nemirovski@uni-wuppertal.de}

\author[Reppekus]{Josias Reppekus}
\address{Universiteit van Amsterdam, the Netherlands}
\email{j.j.reppekus@uva.nl}

\author[Shcherbina]{Nikolay Shcherbina$^\dagger$}
\thanks{$^\dagger$ Nikolay Shcherbina passed away on February 7, 2026.}
\address{Bergische Universit\"at Wuppertal, Germany}

\begin{abstract}
The paper is concerned with the boundary behaviour 
of holomorphically and meromorphically convex hulls 
in pseudoconvex domains in Stein manifolds. 
As an application, it is shown that every connected 
holomorphically or meromorphically convex compact set 
with $C^1$ boundary in a Stein manifold is isotopic to the closure 
of a smoothly bounded strictly pseudoconvex
domain that is also holomorphically or meromorphically convex.
\end{abstract}


\maketitle

\section*{Introduction}

Let $K\Subset\C^n$ be a compact subset. Its {\it polynomially convex hull\/}
is the compact subset
$$
\wh{K} = \bigl\{ z\in \C^n \; \bigl| \bigr.\; |P(z)|\le \|P\|_K 
\text{ for all }P \in\C [z_1,\ldots,z_n] \bigr\}
$$
and its {\it rationally convex hull\/} is the compact subset
$$
\rh{K} = 
\left\{ 
z\in \C^n 
\left|\, 
|R(z)|\le \|R\|_K \,
\begin{array}{c}
\text{for all } R\in\C (z_1,\ldots,z_n)\\ 
\text{that are regular on } K
\end{array}  
\right. \!\!
\right\},
$$
where $\|\cdot\|_K$ denotes the uniform norm of a continuous function on $K$.
The subset $K$ is called {\it polynomially convex\/} if $K=\wh{K}$ and {\it rationally convex\/} if $K=\rh{K}$.
These notions are crucial for polynomial and rational approximation because of the Oka--Weil theorem~\cite[\S 1.5]{Sto}
according to which a function holomorphic on a neighbourhood of a polynomially or rationally convex subset~$K$
can be approximated by polynomials or, respectively, rational functions uniformly on~$K$.

In recent decades, powerful methods originating in symplectic and contact geometry 
have been applied to the study of polynomially and rationally
convex sets, see~\cite{CE1, CE3, For}. 
They are particularly effective in complex dimension~$2$, 
producing results already at the level 
of the underlying differential topology.
However, these techniques require the compact set $K$ 
to be the closure of a strictly pseudoconvex domain, 
which entails, in particular, 
that its boundary is at least $C^2$-smooth. 
On the other hand, the questions about
the differential topology of $K$ make perfect sense if 
its boundary is only $C^1$ and, perhaps more importantly, 
if it is not assumed strictly pseudoconvex. 
Part of the motivation for the present paper 
has been to bridge this gap.

Our main result in this direction, Theorem~\ref{spsc-approx}, 
holds in the more general setting of holomorphically
and meromorphically convex sets in Stein manifolds.
Specialising it to the case of $\C^n$, we obtain the
following statement.

\begin{mainthm}
\label{spsc-approx-cn}
Every polynomially or, respectively, rationally convex 
connected compact set $K\Subset\C^n$ with $C^1$~boundary 
is ambiently isotopic to the closure of a smoothly bounded 
strictly pseudoconvex domain that is polynomially 
or, respectively, rationally convex.
\end{mainthm}

In fact, there is a deformation class of such domains canonically
associated to $K$ by Remark~\ref{spsc-connected}, which 
should in principle allow one to define various symplectic and
contact invariants for~$K$ itself. As a straightforward 
application of Theorem~\ref{spsc-approx-cn}, 
we strengthen the results from~\cite{NS} and~\cite{MT},
see Theorem~\ref{NS-MT-gen}. 

The first step in the proof of Theorem~\ref{spsc-approx}
is a construction of strictly pseudoconvex isotopic
deformations for pseudoconvex domains with $C^1$~boundary
in Stein manifolds. 
The classical approach based on the plurisubharmonicity
of the $-\log$ of the boundary distance function 
is not directly applicable 
because this function may be non-differentiable 
on a dense subset of the domain~\cite{Sa}.
Nevertheless, it retains enough
regularity at the boundary for a Richberg-type 
regularisation scheme to produce a $C^1$ defining function
with strictly pseudoconvex level sets, 
see Proposition~\ref{def-func} 
and the discussion preceding it in~\S\ref{subsec-spsc}.

The second step is to show that for a domain 
with holomorphically or  
meromorphically convex closure, the closures of 
the strictly pseudoconvex subdomains produced
in the first step will be holomorphically 
or, respectively, meromorphically convex. 
The main tool is the local maximum principle
for plurisubharmonic functions on the `added' hulls. 
The well-known case of holomorphically convex hulls
is recalled in~\S\ref{subsec-lmp-psh-hulls}.
The exact analogue for meromorphically convex hulls
has not appeared in the literature;
two proofs are given 
in~\S\ref{subsec-lmp-rat-hulls}
and Appendix~\ref{app-Rossi}.

The argument in the convexity part 
of the proof of Theorem~\ref{spsc-approx}
suggests the following general problem:

\smallskip
\noindent
{\it Suppose that the holomorphically or meromorphically convex hull
of a compact set is contained in the closure of
a pseudoconvex domain. What can be said about 
its intersection with the boundary of the domain\/}?

\smallskip
A reasonable setting for this question is provided by
`domains with simple boundary' that are subgraphs 
of continuous real functions in some local holomorphic coordinates,
see Definition~\ref{simple-boundary}.
In addition to the local maximum principles,
two more ingredients are available.
First, there is a Mergelyan-type approximation theorem
for continuous plurisubharmonic functions,
see \cite{AHP} and Proposition~\ref{psh-Mergelyan}.
Second, the existence of plurisubharmonic bounded 
exhaustion functions (known as `hyperconvexity')
was recently established in full generality
with the help of AI tools in~\cite{WY}, see~Proposition~\ref{simple-hyper}.
The resulting Theorem~\ref{mainhull-n} specialises
to the next statement on polynomially and rationally convex hulls in $\C^n$.
(The case $n=1$ is elementary.)

\begin{mainthm}
\label{mainthmB}
Let $\Omega\subset\C^n$ be a bounded pseudoconvex domain with simple boundary
and $K\Subset\Omega$ a compact subset. 
If the polynomially or rationally convex hull of $K$ is contained in
the closure of $\Omega$, then it is contained in $\Omega$,
i.e., does not intersect $\p\Omega$.
\end{mainthm}

A more precise version of Theorem~\ref{mainhull-n} 
in terms of convex hulls with respect to 
plurisubharmonic functions in $\Omega$
continuous on $\conj{\Omega}$ is given 
by Proposition~\ref{boundary-hull-n}. 
Hulls of this type are recalled and 
discussed 
in~\S\ref{subsec-lmp-psh-hulls}.
For domains with simple boundary in Stein manifolds, 
they are related to holomorphically and 
meromorphically convex hulls 
by Proposition~\ref{hulls-in-pOhulls}
and Corollary~\ref{phull=pOhull}.

In dimension two, considerably stronger geometric results  
can be obtained from the work of the third author~\cite{S1}
recalled in Section~\ref{sec-p-hulls-graphs-c2}.
This technique is used to prove Theorem~\ref{foliation} 
applicable to various hulls contained in closures of
pseudoconvex domains with simple boundary on
complex surfaces. For polynomial and rational hulls 
in~$\C^2$, we have the following statement.

\begin{mainthm}
Let $\Omega$ be a pseudoconvex domain with simple boundary in $\C^2$.
If the polynomially or rationally convex hull
of a compact set $K$ is contained in the closure of $\Omega$, 
then the intersection of this hull with $\p\Omega-K$
is a union of analytic discs.
\end{mainthm}

It was shown by Catlin~\cite{Ca} that such an assertion 
cannot hold in higher dimensions, see Remark~\ref{Catlin}.

\subsection*{Conventions}
`Smooth' without qualifiers means~$C^\infty$.
A `domain' is a connected open subset. The terms
`Stein' and `(locally) pseudoconvex' for domains 
in Stein manifolds are used as synonyms  
by the solution of the Levi problem~\cite{DG}.
`Plurisubharmonic' is abbreviated to psh and `subharmonic' to sh. 
A psh function on a non-open subset is a psh function
on a (non-specified) open neighbourhood of that subset. 

\subsection*{Acknowledgements}
The authors thank Alberto Abbondandolo and Tobias Harz
for helpful discussions and Alexander Izzo for 
insightful comments and the short proof 
of Lemma~\ref{Stein-boundary} in~$\C^n$. 

\subsection*{AI declaration} The authors did not use AI tools.

\section{Local maximum principles}

\subsection{Local maximum principle on holomorphic and psh hulls}
\label{subsec-lmp-psh-hulls}
Let $M$ be a complex manifold. The {\it holomorphically convex hull\/}
of a compact subset $K\Subset M$ is the set
$$
\oh{K} = \{p\in M\mid |f(p)|\le \|f\|_K \text{ for all } f\in\OO(M)\}.
$$
$K$ is said to be {\it holomorphically convex\/} in $M$ if $K=\oh{K}$.

If $M=\C^n$, then 
\begin{equation}
\label{pol=hol}
\oh{K}=\wh{K}
\end{equation}
for every $K\Subset\C^n$ because every $f\in\OO(\C^n)$ can be 
approximated by polynomials uniformly on compact subsets.

It is a classical fact that if the manifold $M$ is Stein, 
then $\oh{K}$ coincides with the hull of $K$
with respect to (continuous) psh functions on~$M$. 
The result essentially goes back to Oka (cf.~the italicised statement 
in the middle of p.~143 in~\cite{O}). Its proofs
for $M=\C^n$ may be found in \cite[Theorem 1.3.11]{Sto} 
or~\cite[Theorem 4.3.4]{H}. 
The general case may be 
obtained by taking a holomorphic embedding $F:M\hookrightarrow \C^n$ 
provided by~\cite[Theorem 5.3.9]{H}
and observing that 
every psh function on $F(M)$ extends to
a psh function on $\C^n$ by~\cite[Theorem 3.2]{BL}.

This coincidence of holomorphic and psh hulls was used  
by Rosay~\cite{R} (see also~\cite[pp.~78--79]{Sto} 
and the proof of Theorem~\ref{max-pOhull} below)
to derive the following local maximum principle
for psh functions on holomorphically convex hulls.

\begin{thm}
\label{max-phull}
Let $M$ be a Stein manifold. 
Suppose that $K\Subset M$ is compact and $U\subset M$ is open with $U\cap K=\varnothing$.
For every psh function $\phi$ on $\conj U$ and any point $p\in U\cap\oh{K}$ we have
\begin{equation}
\label{max-phull-ineq}
\phi(p) \le \max\limits_{\p U\mathop{\cap}\oh{K}} \phi.
\end{equation}
\end{thm}

Applying the theorem to moduli of polynomials in $\C^n$ recovers Rossi's 
local maximum principle~\cite{Ro} for polynomially convex hulls:
\begin{equation}
\label{local-phull-phull}
\conj{U} \cap  \wh{K} \subseteq \wh{\p U\cap \wh{K}}
\end{equation}
for all open $U\subset\C^n$ such that $U\cap K=\varnothing$.

Rosay's argument can be used {\it mutatis mutandis\/} for the following
more general class of psh hulls.

\begin{df}
Let $\Omega$ be a domain in an arbitrary complex manifold and $K\Subset \conj{\Omega}$  be 
a compact subset of its closure. The $P(\conj{\Omega})$-hull of $K$
is the subset
$$
\ph{K}:=\{ p\in \conj{\Omega} \mid \phi(p)\le \max\limits_K \phi \text{ for all } \phi\in P(\conj{\Omega})\},
$$
where $P(\conj{\Omega})=PSH(\Omega)\cap C^0(\conj{\Omega})$ is the cone of continuous
functions on $\conj{\Omega}$ that are psh in~$\Omega$.
\end{df}

\begin{thm}
\label{max-pOhull}
Let $K\Subset\conj{\Omega}$ be such that $\ph{K}$ is compact
and let $U$ be open with $U\cap K=\varnothing$.
Let $\phi$ be a function continuous on $U'\cap\conj{\Omega}$ and psh on $U'\cap\Omega$
for an open set $U'\supset \conj{U}$. Then
\begin{equation}
\label{max-pOhull-ineq}
\phi(p) \le \max\limits_{\p U\mathop{\cap}\ph{K}} \phi
\end{equation}
for every $p\in U\cap\ph{K}$.
\end{thm}

\begin{proof} 
The argument is by contradiction. Suppose that \eqref{max-pOhull-ineq} is false.
Adding a constant, we may assume that $\phi(p)>0$ whereas $\phi<-\eps$ 
for some $\eps>0$ on $W\cap\conj{\Omega}$ for a neighbourhood $W\supset \p U\cap \ph{K}$.

There exists a non-negative function $\psi\in P(\conj{\Omega})$ vanishing precisely on $\ph{K}$.
This is proven exactly as in~\cite[\S 2]{R} by taking a weighted sum of functions in $P(\conj{\Omega})$
vanishing on $\ph{K}$ and positive at a given point in its complement.

Let $V$ be the intersection of $U$ with a relatively compact neighbourhood of $\ph{K}$.
There exists $C>0$ such that $\phi < -\eps + C\psi$ on the compact set 
$(\p V - W)\cap\conj{\Omega}\Subset \conj{\Omega}-\ph{K}$.
Then 
$$
\wt{\phi}:=
\left\{ 
\begin{array}{ll}
\max\{\phi, -\eps + C\psi\} & \text{ on } V\cap\conj{\Omega},\\
- \eps + C\psi & \text{ on } \conj{\Omega}-V
\end{array}
\right.
$$
is a well-defined function in $P(\conj{\Omega})$ 
that is equal to $-\eps<0$ on $K$ and to $\phi(p)>0$ at $p\in\ph{K}$,
which contradicts the definition of the $P(\conj{\Omega})$-hull.
\end{proof}

\subsection{Local maximum principle on meromorphic hulls}
\label{subsec-lmp-rat-hulls}
Let us define the {\it meromorphically convex hull\/} 
of a compact subset $K$ in a complex manifold $M$ as the set
$$
\mh{K}=\left\{ p\in M 
\left|\,
|f(p)|\le \|f\|_K 
\begin{array}{c}
\text{\small{for all }}
f \text{\small{ meromorphic on }} M\\ 
\text{\small{and regular on }} K
\end{array}  
\right. \!\!
\right\}.
$$
$K$ is said to be {\it meromorphically convex\/} in $M$
if $K=\mh{K}$.

Note that our definition entails that if $f$ is
regular on~$K$, then it is regular on~$\mh{K}$. 
Hence, $\mh{\mh{K}}=\mh{K}$ whenever $\mh{K}$ is also compact. 

If $f$ is a meromorphic function on a complex manifold $M$,
then there exist a covering of $M$ by open sets $U_\alpha$
and holomorphic functions $s_\alpha, t_\alpha\in\OO(U_\alpha)$
with mutually prime germs at every $p\in U_\alpha$
such that $f=s_\alpha/t_\alpha$ in $U_\alpha$,
see~\cite[\S 6.2]{H}.
In particular, $f$ is regular at $p\in U_\alpha$ 
if and only if $t_\alpha(p)\ne 0$.
Furthermore, $g_{\alpha\beta}:=s_\alpha/s_\beta=t_\alpha/t_\beta$
are non-vanishing holomorphic functions on $U_\alpha\cap U_\beta$
satisfying the multiplicative cocycle condition. 
If $L_f$ is the holomorphic line bundle on $M$ with
transition functions $g_{\alpha\beta}\in\OO^*(U_\alpha\cap U_\beta)$,
then $\{s_\alpha\}$ and $\{t_\alpha\}$ define holomorphic
sections $s,t\in\OO(L_f)$ such that $f=s/t$ and $f$ is regular precisely
on the set where~$t\ne 0$. 

Suppose that $p\notin\mh{K}$. 
There exists a meromorphic function
$f=s/t$ with $s,t\in\OO(L_f)$ such that $t\ne 0$ on $K$
and $f$ violates the condition 
in the definition of the meromorphic hull at~$p$. 
Therefore, 
$$
\tau = 
\left\{\, 
\begin{array}{ll}
t & \text{ if } t(p)=0, \\ 
s - f(p)t &  \text{ if } t(p)\ne 0
\end{array}  
\right.
$$
is a holomorphic section of $L_f$ vanishing at $p$ and 
non-vanishing on~$K$. 
In other words, through every point $p\notin\mh{K}$, 
there passes an analytic hypersurface in $M$ avoiding~$K$.

If $M$ is Stein, then we may also find a section $\sigma\in\OO(L_f)$
such that $\sigma(p)\ne 0$ by Cartan's Theorem~A.
The meromorphic function $\sigma/\tau$ has a pole at $p$
and therefore violates the condition in the definition 
of the meromorphic hull on a neighbourhood of~$p$.
This shows that $\mh{K}$ is closed and, since it
is contained in $\oh{K}$, compact for every compact
subset $K$ in a Stein manifold.

If $M=\C^n$, the line bundle $L_f$ is trivial, 
cf.~\cite[Theorem~7.4.5]{H}.
Thus, for every $p\notin\mh{K}$, 
there is an entire function $\tau$ 
such that $\tau(p)=0$ and $\tau\ne 0$ on~$K$. 
A sufficiently long Taylor polynomial $Q$ for $\tau$ at~$p$
satisfies $Q(p)=0$ and $Q\ne 0$ on~$K$.
The rational function $1/Q$ is regular on $K$
and has a pole at~$p$, whence $p\notin\rh{K}$. 
Hence, $\rh{K}\subseteq \mh{K}$
and, since the opposite inclusion is obvious, we get the
well-known fact that
\begin{equation}
\label{rat=mer}
\mh{K} = \rh{K}
\end{equation}
for every compact subset $K\Subset\C^n$.

We may now state a version of the local maximum principle 
for psh functions on meromorphically convex hulls in Stein manifolds
and, in particular, on rationally convex hulls in $\C^n$.

\begin{thm}
\label{max-rhull}
Let $M$ be a Stein manifold. 
Suppose that $K\Subset M$ is compact and $U\subset M$ is open with $U\cap K=\varnothing$.
For every psh function $\phi$ on $\conj U$ and any point $p\in U\cap \mh{K}$, we have
\begin{equation}
\label{max-rhull-ineq}
\phi(p) \le \max\limits_{\p U\mathop{\cap}\mh{K}\strut} \phi.
\end{equation}
\end{thm}

Our first proof of Theorem~\ref{max-rhull} is close in spirit to
Rosay's proof of Theorem~\ref{max-phull}.  
It relies on a description of meromorphically convex hulls 
in terms of psh functions obtained originally 
for rationally convex hulls in $\C^n$ by Duval and Sibony~\cite{DS} 
and generalised to Stein manifolds by Guedj~\cite{Gu}.
We use the statement from~\cite[Theorem 1.3]{BGS},
cf.\ also~\cite{N1}.

\begin{thm}
\label{rhull-psh}
A compact subset $K$ of a Stein manifold~$M$ 
is meromorphically convex
if and only if it has a neighbourhood basis 
such that every element of the basis is of the form $U_\rho=\{\rho < 0\}$, 
where $\rho$ is strictly psh on~$\conj{U_\rho}$
and $dd^c\rho$ extends to a Hodge form on~$M$.
\end{thm}

Two remarks are in order. First, a {\it Hodge form\/} on $M$
is a K\"ahler form such that its de Rham cohomology class 
is integral, i.e.\ lies in the image of the homomorphism 
$\mathrm{H}^2(M;\Z)\to\mathrm{H}^2(M;\R)$. 
This condition cannot be omitted in general, see~\cite{Sh}.
Second, the result in~\cite{BGS} is stated for compact
sets that are {\it hypersurface convex\/} in $M$ 
in the sense that there exists an analytic hypersurface 
avoiding $K$ through every point~$p\notin K$. 
For subsets of Stein manifolds,
this notion is equivalent to meromorphic convexity. 
Indeed, we have already seen that if $K=\mh{K}$, 
then $K$ is hypersurface convex for any~$M$.
Conversely, assume that $K$ is hypersurface convex
and let $p\notin K$. 
Take a hypersurface through~$p$ avoiding $K$; 
it is the zero set of a holomorphic section $\tau$ of a line bundle~$L$.
By Cartan's Theorem~$A$, there exists a section $\sigma\in\OO(L)$
such that $\sigma(p)\ne 0$. The meromorphic function $\sigma/\tau$
is regular on $K$ and has a pole at $p$, which shows that
$p\notin\mh{K}$. 

\begin{proof}[Proof of Theorem\/~{\rm \ref{max-rhull}}]
Let us suppose that \eqref{max-rhull-ineq} is false and argue by contradiction. 
Adding a constant to $\phi$, we may suppose that $\phi(p)>0$ whereas $\phi<0$
on a neighbourhood $W\supset \p U\cap \mh{K}$. 
$U$~may be assumed relatively compact in $M$ by intersecting 
it with a relatively compact neighbourhood of $\mh{K}$.
We may further assume that $\phi$ is smooth, e.g.,
by embedding $M$ in $\C^n$, taking a holomorphic
retraction $\pi$ of a neighbourhood of $M\subset\C^n$ onto $M$
as in~\cite[Theorem VIII.C.8]{GR}, and regularising
the psh function $\phi\circ\pi$ on an open neighbourhood
of $\conj{U}$ in $\C^n$ in the usual way
by convolution~\cite[\S 2.6]{H}.

Let $\eta:\R\to \R$ be a smooth convex increasing function 
such that $\eta(x)= 0$ for $x\le 0$ and $\eta(x)>0$ for $x>0$. 
Set $\phi_1 = \eta\circ\phi$. Then
$\phi_1$ is a smooth non-negative psh function on $\conj{U}$ 
that is identically zero on~$W$ and positive at~$p$.

Next choose a compactly supported smooth function $\lambda\ge 0$ on $M$
such that $\lambda\equiv 1$ on a neighbourhood $V\supset\mh{K}$ 
and $\lambda\equiv 0$ on a neighbourhood of $\p U - W$.
(The latter set is closed and disjoint from $\mh{K}$.)
Define $\phi_2 = \lambda\cdot\phi_1$. 
Then $\phi_2$ is identically zero on a neighbourhood of $\p U$
and can be smoothly extended to $M$ by~zero.
Note that $\phi_2\ge 0$ on $M$, $\supp\phi_2\Subset U$, 
$\phi_2(p)>0$, and $\phi_2$ is psh on the complement 
of the {\it compact\/} subset $N:=\supp\lambda - V\Subset M$. 

By the `only if' part of Theorem~\ref{rhull-psh}, there exists a neighbourhood 
$$
\mh{K} \Subset U_\rho \subset \conj{U_\rho} \Subset V
$$
such that $dd^c\rho$ 
extends to a Hodge form $\omega_\rho$ on $M$.
In particular, $\omega_\rho$ is positive on the compact 
set $N$ defined above. Hence, there exists a positive integer $C\in\mathbb{N}$ 
such that $dd^c\phi_2 + C\omega_\rho$ is a Hodge form on $M$.

Let $\theta:\R\to \R$ be a smooth convex and strictly increasing 
function such that $\theta(x)>-\phi_2(p)/C$ for $x<0$
and $\theta(x)= x$ for $x\ge 0$.
Then $\rho_1:=\theta\circ\rho$ is a strictly psh function near $\conj{U_\rho}$.
It equals $\rho$ outside of $U_\rho$
and satisfies $\phi_2(p)+C\rho_1(p)>0$.

Now $\rho_2:=\phi_2+C\rho_1$ is a strictly psh function
on a neighbourhood of $\conj{U_\rho}$ 
such that $U_{\rho_2}\subset U_\rho$ and 
$dd^c\rho_2$ extends to a Hodge form on $M$
by setting it equal to $dd^c\phi_2 + C\omega_\rho$
on $M-\conj{U_\rho}$.
Hence, $\conj{U_{\rho_2}}$ 
is meromorphically convex by the `if' part of Theorem~\ref{rhull-psh}.
On the other hand, $K\subset U_{\rho_2}$ and $p\notin\conj{U_{\rho_2}}$,
which contradicts $p\in\mh{K}$. 
\end{proof}

\begin{rem}
Applying Theorem~\ref{max-rhull} to moduli of global holomorphic functions
on~$M$, we obtain the local inclusion of hulls
\begin{equation}
\label{local-rhull-phull}
 \conj{U} \cap \mh{K} \subseteq \oh{\p U\cap \mh{K}}
\end{equation}
for all open $U\subset M$ such that $U\cap K=\varnothing$.
Note that we may not take the meromorphically convex hull 
of $\p U\cap \mh{K}$ on the right-hand side of~\eqref{local-rhull-phull}.
An example of a compact set $K\Subset\C^n$ such that
\begin{equation}
\label{local-rhull-rhull}
\conj{U} \cap \rh{K} \nsubseteq \rh{\p U\cap \rh{K}}
\end{equation}
was given by Stolzenberg~\cite[Example 1.11]{St}.
\end{rem}

\begin{rem}
\label{strict-max}
A useful observation is that for a {\it strictly\/} psh function~$\phi$
the inequalities in Theorems~\ref{max-phull} and~\ref{max-rhull} must be strict.
Indeed, one can find a psh function $\wt{\phi}$ such that $\wt{\phi}(p)=\phi(p)$
and $\wt{\phi}<\phi$ on $\conj{U}-\{p\}$.
\end{rem}

\section{Domains with simple boundary}
\label{sec-simple-boundary}

\subsection{Definition and general properties}
\label{subsec-simple-boundary}
We will consider a class of domains with sufficiently nice 
$C^0$-boundaries in complex manifolds.

\begin{df}
\label{simple-boundary}
A domain $\Omega$ has {\it simple boundary\/} if for every $p\in\p\Omega$
there exist local holomorphic coordinates $(z_1,...,z_{n-1},w)$, $w=u+iv$, 
centred at $p$ and a neighbourhood of $p=(0,\ldots,0)$ of the form 
$$
U = B\times I,
$$
where $B\subset\C_z^{n-1}\times\R_u$ is an open Euclidean ball and $I=(v_-,v_+)\subset\R_v$ is an open interval,
such that
$$
\Omega\cap U = \{ (z,w)\in U \mid v < F(z,u) \} 
$$
for a continuous function $F:\conj{B}\to I$.
\end{df}

\begin{rem} 
\label{sb-properties}
Note the following easy consequences of the definition:
\begin{enumerate}
\item A domain with $C^k$-boundary, $1\le k\le\infty$, 
has simple boundary by the implicit function theorem.
\item $\Omega=\mathring{\conj{\Omega}}$.
\end{enumerate}
\end{rem}

Another immediate corollary of Definition~\ref{simple-boundary} is  
a biholomorphically invariant version of the {\it segment property\/}:
For every boundary point $p\in\p\Omega$, there exist a non-vanishing
holomorphic real vector field $X$ near $p$ and a neighbourhood $W\ni p$ such that 
$e^{\eps X}(W\cap\conj{\Omega})\subset\Omega$
for all $\eps>0$ small enough. (Just take $X=-\p/\p v$ for 
a coordinate system from Definition~\ref{simple-boundary}.) 
Conversely, it follows from~\cite[Theorem 3.3]{F}
that a domain with this property has simple boundary.
 
The following Mergelyan-type theorem for continuous psh
functions on bounded domains with simple boundary in~$\C^n$
was obtained in~\cite{AHP}.
(The case of $C^1$-smooth boundaries was treated
much earlier in~\cite{FW}.) 
For convenience, we state it for not necessarily bounded domains
in Stein manifolds.

\begin{prop}
\label{psh-Mergelyan}
If $\Omega$ is a domain with simple boundary in a Stein manifold~$M$,
then functions from $P(\conj{\Omega})$ can be approximated
uniformly on compact subsets of $\conj{\Omega}$ by continuous psh 
functions on {\rm (}varying\/{\rm )} neighbourhoods of $\conj{\Omega}$.
\end{prop}

\begin{proof}
As observed in~\cite[\S 2]{PW}, it follows from the segment
property that a function $\phi\in P(\conj{\Omega})$
satisfies the assumptions of Gauthier's localisation theorem~\cite[\S 3]{Gau}.
By Corollary~1 in~\cite[\S 3]{Gau}, for every $\eps>0$, 
there exists a continuous psh function $\psi$ on $\conj{\Omega}$ such that
$$
|\psi(p) - \phi(p)| < \eps\left(|F(p)|^2+1\right), \quad p\in\conj{\Omega},
$$
where $F:M\to\C^n$ is a holomorphic embedding,
which proves the proposition.
\end{proof}

Combining the Mergelyan property with the local maximum principles 
for psh functions allows us to compare different hulls in~$\conj{\Omega}$.

\begin{prop}
\label{hulls-in-pOhulls}
If the holomorphically or meromorphically convex hull $\wt{K}$ of
a compact set $K$ in a Stein manifold is contained in the
closure $\conj{\Omega}$ of a domain with 
simple boundary, then $\wt{K}\subseteq\ph{K}$.
\end{prop}

\begin{proof}
Suppose that $p\in\wt{K}-K$ and $\phi\in P(\conj{\Omega})$.
Let $\phi_m$, $m\in\mathbb{N}$, be a sequence of continuous psh functions 
on the closures of open sets $V_m\supset \conj{\Omega}$ converging
uniformly to $\phi$ on $\wt{K}$ provided by
Proposition~\ref{psh-Mergelyan}. Applying Theorem~\ref{max-phull}
or Theorem~\ref{max-rhull} to the compact set $K$, the open set $U_m:=V_m-K$,
and the psh function $\phi_m$ on $\conj{U}_m$, we obtain
$$
\phi_m(p)\le \max\limits_{\p U_m\mathop{\cap}\wt{K}} \phi_m \le \max\limits_{K} \phi_m
$$
for all $m\in\mathbb{N}$. Letting $m\to\infty$, we conclude that 
$$
\phi(p)\le \max\limits_K \phi.
$$
Hence, $p\in\ph{K}$.
\end{proof}

\begin{cor}
\label{phull=pOhull}
If $\Omega$ has simple boundary in a Stein manifold and $\oh{K}$ is contained in $\conj\Omega$,
then $\oh{K}=\ph{K}$.
\end{cor}

\begin{proof}
$\oh{K}\subseteq\ph{K}$ by the preceding proposition and the opposite
inclusion is obvious because moduli of holomorphic functions are in~$P(\conj{\Omega})$.
\end{proof}

For a domain $\Omega$ with simple boundary, functions from $P(\conj{\Omega})$
are psh on complex analytic subsets in~$\conj{\Omega}$. 
This observation appeared 
in a slightly different form in~\cite[\S 4.B]{KR}.

\begin{cor}
\label{psh-on-bd-disc}
If $\Omega$ has simple boundary and $\phi\in P(\conj{\Omega})$, 
then the restriction of $\phi$ to an analytic disc $\Delta\subset\conj{\Omega}$ 
is subharmonic.
\end{cor}

\begin{proof}
${\phi|}_\Delta$ is the local uniform limit of sh functions by Proposition~\ref{psh-Mergelyan}.
These local approximations may also be constructed directly using the segment property.
Assume w.\,l.\,o.\,g.\ that $\Delta$ lies in a coordinate neighbourhood from Definition~\ref{simple-boundary}.
Then ${\phi|}_\Delta$ is the uniform limit  
of sh functions $\phi_m(z,w):=\phi(z,w-i/m)$ as $m\to\infty$.
\end{proof}

\subsection{Pseudoconvex domains}
The following proposition implies that the closure of a pseudoconvex domain 
with simple boundary is locally holomorphically convex. 
For domains with $C^1$-smooth boundary, this local result 
was obtained in~\cite[Lemme 3]{HS} by a different argument.

\begin{prop}
\label{loc-pol-conv}
Let $F:\conj{B}\to \R_v$ be a continuous function 
on a Euclidean ball in $\C^{n-1}_z\times\R_u$ such that the domain
$$
\Omega_F := \{(z,w)\in\C^n \mid (z,u)\in B, \, v<F(z,u) \}
$$
is pseudoconvex. If $F\ge v_-$ on $\conj{B}$,  then the set 
$$
K=K(v_-,F):=\{(z,w)\in\C^n \mid (z,u)\in \conj{B}, \, v_-\le v\le F(z,u)\}
$$
is polynomially convex. {\rm (}Here $w=u+iv$ as in Definition~{\rm \ref{simple-boundary}}.{\rm )}
\end{prop}

\begin{proof}[Proof {\rm (cf.~Lemma~9 in~\cite{S2})}]
The polynomially convex hull is contained in the usual convex hull~\cite[p.~3]{Sto},
so $\wh{K}\subset\conj{B}\times [v_-,+\infty)$.

Note that the function $\rho(z,w)= |z|^2 + u^2$ is {\it strictly\/} psh on $\C^n$.
Hence, the set $\wh{K}-K$ cannot intersect $\p B\times \R_v$ because
then $\rho$ would attain its maximum there, which is impossible by Remark~\ref{strict-max}.

Therefore if $\wh{K}-K\ne\varnothing$, then
$$
v^* := \max\limits_{\wh{K}} \,(v - F(z,u))\, > 0
$$
and the maximum is attained at $q=(z,u,v^*+F(z,u))$ with $(z,u)\in B$.

Let $s_t(z,w):=(z, w+it)$ be the shift in the $v$-variable. 
For a fixed $\eps>0$ consider the compact set
$$
X = \bigcup\limits_{t\in [-\eps,0]} s_t(K) 
$$
and denote by $Y$ its intersection with $\p B\times \R_v$.

Since $Y$ is compactly contained in the {\it strictly\/} pseudoconvex part
of the boundary of $\Omega_{v^*+F}$ disjoint from the graph $\Gamma(v^*+F)$, 
it follows that there exists a neighbourhood $V\supset Y$ with $V\cap\Gamma(v^*+F)=\varnothing$
such that the domain $\wt{\Omega}=\Omega_{v^*+F}\cup V$ is pseudoconvex and $Y\Subset\wt{\Omega}$.
(This observation is known as the ``Grauert bump method'',
see~\cite[IX.B.4]{GR}.) 
Then also $X\Subset\wt{\Omega}$ and hence $s_t(K)\Subset\wt{\Omega}$ for all $t\in [-\eps,0]$.

Furthermore, $s_{t}(\wh{K})\Subset\wt{\Omega}$ for all negative $t\in [-\eps,0)$ 
by the definition of $v^*$ and because $\wh{K}$ and $K$ have the same intersection 
with $\p B\times\R_v$.

Let now $\psi:\wt{\Omega}\to\R$ be a psh exhaustion function on~$\wt{\Omega}$.
Set
$$
m_X:=\max\limits_X \psi.
$$
Since $q\in\Gamma(v^*+F)\subset \p\wt{\Omega}$ and $s_t(q)\in\wt{\Omega}$ for $t<0$, 
there exists a $t_0\in [-\eps,0)$
such that $\psi(s_{t_0}(q))>m_X$ and, in particular,
$$
\psi(s_{t_0}(q)) > \max\limits_{s_{t_0}(K)} \psi.
$$
However, $s_{t_0}(q)\in s_{t_0}(\wh{K}-K)=\wh{s_{t_0}(K)} - s_{t_0}(K)$. Hence,
we obtain a contradiction with Theorem~\ref{max-phull} applied
to the compact set $s_{t_0}(K)$ and the open set $U=U'-s_{t_0}(K)$
for any $U'\supset \wh{s_{t_0}(K)}$ with $\conj{U'}\subset \wt{\Omega}$.

This proves that $\wh{K}=K$ and the proposition follows.
\end{proof}

A related important property of pseudoconvex domains with simple boundary 
was established by Kerzman and Rosay in~\cite{KR}.

\begin{prop}
\label{disc-touching-bd}
If the image of a holomorphic map from the unit disc 
lies in the closure of a pseudoconvex domain 
with simple boundary, then it lies either entirely 
inside the domain or entirely in its boundary.
\end{prop}

This (local) result is proven in~\cite[\S 4.D]{KR} for a pseudoconvex domain
in~$\C^n$ satisfying condition~(*) on p.~172 of~\cite{KR},
and that condition is precisely the local segment property 
recalled in~\S\ref{subsec-simple-boundary}.
The example in~\cite[\S 4.C]{KR} shows that 
the assertion of the proposition can be false
for pseudoconvex domains in $\C^2$ with mildly 
discontinuous boundary.

\subsection{Hyperconvexity}
A domain $\Omega$ in a complex manifold is called {\it hyperconvex\/} 
if there is a continuous psh function $\phi:\Omega\to (-\infty, 0)$
such that its sublevel sets $\{p\in\Omega\mid \phi(p)\le c\}$ are compact for all $c<0$.
(Such functions are called `bounded exhaustion' functions.)
We will always assume that $\phi$ is continuously extended 
to $\p\Omega$ by zero. 

For bounded domains in Stein manifolds, hyperconvexity 
is a local property, see~\cite[\S 0.C]{KR} and~\cite[Theorem 1.6]{Bl}. Namely,
it is enough to assume that for every $p\in\p\Omega$
there exist a neighbourhood $V\ni p$ and a 
(not necessarily continuous) negative
psh function $\phi$ on $\Omega\cap V$ such that 
$\phi(q)\to 0$ as $q\to \p\Omega\cap V$. 

\begin{prop}
\label{simple-hyper}
Every bounded pseudoconvex domain with simple boundary in a Stein manifold
is hyperconvex.
\end{prop}

This result has a long history. If the boundary is of class $C^k$, $k\ge 2$, 
it is due to Diederich and Fornaess~\cite{DF}, 
see Remark~\ref{c-infty} for the precise statement.
Kerzman and Rosay~\cite{KR} treated the case of $C^1$ boundaries by
a method based on translations of the boundary distance function.
Chen~\cite{C} applied estimates of the relative extremal function
to obtain local hyperconvexity when the function $F$ 
in Definition~\ref{simple-boundary} is H\"older. 

Finally, Wang and Yang~\cite{WY} relied on AI tools 
to settle the general case using Poletsky's theory of disc envelopes. 
In the follow-up paper by Harz~\cite{Har},
it was shown that the largest psh minorant of $v-F(z,u)$
in $\Omega\cap U$ provides a negative psh function tending
to zero at~$\p\Omega\cap U$. 
The key input from~\cite{WY} is that the continuity principle
can be applied to the holomorphic discs in the Poletsky description
of this minorant after an elementary holomorphic perturbation 
along the $w$-axis.

\begin{rem}
As noted in~\cite[\S 4.B]{KR}, Proposition~\ref{disc-touching-bd}
follows readily from the local form of Proposition~\ref{simple-hyper}
and Corollary~\ref{psh-on-bd-disc}.
\end{rem}

\section{Polynomial hulls of graphs in $\C^2$}
\label{sec-p-hulls-graphs-c2}

The material in this section is mostly taken from~\cite{S1} and~\cite{S2}.

\subsection{Analytic discs with boundary}
For the purposes of this paper we adopt the following definition
of an analytic disc with boundary in a given set, which includes
the discs constructed in~\cite{S1}.

\begin{df}
\label{disc-w-boundary}
Let $M$ be a complex manifold and $X$ an arbitrary subset of~$M$. 
An {\it analytic disc $\Delta$ in $M$ with boundary in $X$}
is the image of an injective holomorphic immersion $\iota:\mathbb{D}\hookrightarrow M$
of the open unit disc $\mathbb{D}=\{|\zeta|<1\}\subset\C$ such that 
\begin{itemize}
\item[1)] the closure $\conj{\Delta}=\conj{\iota(\mathbb{D})}$ is compact in~$M$,
\item[2)] the boundary cluster set $b\Delta:=\bigcap\limits_{K\Subset\mathbb{D}} \conj{\iota(\mathbb{D}-K)}$ lies in $X$.
\end{itemize}
\end{df}

Note that if $\Delta\cap X=\varnothing$, then $\Delta$ is closed in $M-X$. In particular, if $X$ is closed
and $\Delta\cap X=\varnothing$, then
$\Delta$ is a (properly embedded) complex submanifold of $M-X$.

\begin{lem}
\label{psh-on-discs}
If $\phi$ is psh on a neighbourhood of $\conj{\Delta}$ and $p\in\Delta$, then
$$
\phi(p)\le \max\limits_{b\Delta} \phi.
$$
If $\phi$ is strictly psh, the inequality is strict.
\end{lem}

The lemma implies, in particular, that $\Delta\subseteq\oh{b\Delta}$ in~$M$.

\begin{proof}
The function $\phi\circ\iota$ is sh on $\mathbb{D}$. Let $\zeta\in\mathbb{D}$ be such that $p=\iota(\zeta)$.
By the usual maximum principle, there is a sequence of
points $\zeta_k\to\partial\mathbb{D}$ 
such that $\phi\circ\iota(\zeta_k)\ge \phi\circ\iota(\zeta)$.
The sequence $\iota(\zeta_k)$ has a subsequence converging to $p_\infty\in\conj{\Delta}$
and $p_\infty\in b\Delta$ by the definition of $b\Delta$. 
Hence, 
$$
\phi(p)
\le \limsup\limits_{k\to\infty} \phi\circ\iota(\zeta_k)\le \phi(p_\infty)
\le \max\limits_{b\Delta}\phi.
$$
If $\phi$ is strictly psh and the equality is attained at $p=\iota(\zeta)\in\Delta$,
then the strictly sh function $\phi\circ\iota$ attains its maximum at an interior
point of $\mathbb{D}$, which is impossible (e.g., by Remark~\ref{strict-max}).
\end{proof}

\subsection{Dirichlet problem for Levi-flat graphs}
In $\C^2$ with coordinates $(z=x+iy,w=u+iv)$ consider the (open) cylinder
$$
\mathfrak{C} := B\times\R_v = \{(z,w)\in\C^2 \mid (z,u)\in B\}
$$
over an (open) Euclidean $3$-ball $B\subset\C_z\times\R_u$. 
Then
$$
\p\mathfrak{C} = S\times\R_v,
$$
where $S=\p B$ is a round $2$-sphere in $\C_z\times\R_u$.

\begin{thm}
\label{Nikolay}
Let $\phi:S\to \R_v$ be continuous and $\Gamma(\phi)\subset\p\mathfrak{C}$ its graph.  

\smallskip
\noindent
{\bf \;(i)} $\wh{\Gamma(\phi)} = \Gamma(\wh{\phi})$ for a continuous function $\wh{\phi}:\conj{B}\to\R_v$.

\smallskip
\noindent
{\bf (ii)} $\wh{\Gamma({\phi})}\cap\mathfrak{C}$ is a  
union of analytic discs with boundary in $\Gamma(\phi)$.
\end{thm}

This is a simplified version of the Main Theorem on p.~478 of~\cite{S1}.
The (easier) first assertion was proven earlier in \cite{A} and \cite{ST}.
If $\phi$ is smooth and sufficiently generic,
assertion (ii) follows from~\cite{BK} or~\cite{K}.

\begin{exm}
\label{hull-const}
If $\phi\equiv v_0\in\R_v$ is constant, then the polynomially convex hull of $\Gamma(\phi)$
is the Levi-flat ball $\conj{B}\times\{v_0\}$ and $\wh{\phi}\equiv v_0$.
\end{exm}

\subsection{Monotonicity of hulls and propagation of intersections}
The hulls described by Theorem~\ref{Nikolay} have a number of useful geometric properties.

\begin{lem}
\label{psc-contain}
Let $F:\conj{B}\to\R_v$ be a continuous function such that the domain
$$
\Omega_F := \{(z,w)\in\mathfrak{C} \mid v<F(z,u) \}
$$
is pseudoconvex. If $\phi\le F$ on $S$, then $\wh{\phi}\le F$ on $\conj{B}$.
\end{lem}

\begin{proof} 
Follows immediately from Proposition~\ref{loc-pol-conv} and $\wh{\Gamma(\phi)} = \Gamma(\wh{\phi})$.
\end{proof}

\begin{cor}
\label{hull-monoton}
If $\phi_-\le\phi_+$ on $S$, then $\wh{\phi}_-\le\wh{\phi}_+$ on $B$.
\end{cor}

\begin{proof}
$\Gamma(\wh{\phi}_+)\cap \mathfrak{C}$ is a union of analytic discs and hence 
$\mathfrak{C}-\Gamma(\wh{\phi}_+)$ is pseudoconvex. So Lemma~\ref{psc-contain}
applies with $F=\wh{\phi}_+$ and $\phi=\phi_-$.
\end{proof}

\begin{cor}
\label{between-graphs}
If $\phi\ge v_-$ on $S$, then the compact set
$$
K(v_-,\phi):=\{(z,w)\in\C^2\mid (z,u)\in\conj{B}, v\in [v_-,\wh{\phi}(z,u)]\}
$$
is polynomially convex.
\end{cor}

\begin{proof}
The complement to $\Gamma(\wh{\phi})$ is pseudoconvex and $\wh{\phi}\ge v_-$ on $B$,
e.g., by Corollary~\ref{hull-monoton}. Hence, the result follows from Proposition~\ref{loc-pol-conv}.
\end{proof}

\begin{lem}
\label{Rouche}
Let $j:\Sigma\to\mathfrak{C}$ be a holomorphic map from a connected 
Riemann surface. 
Suppose that $j(\Sigma)\subset \{ (z,w)\in\mathfrak{C}\mid v\le \wh{\phi}(z,u)\}$
for a continuous function $\phi:S\to\R_v$. 
If $j(\Sigma)\cap\Gamma(\wh{\phi})\neq \varnothing$,
then $j(\Sigma)$ is contained in an analytic disc in $\mathfrak{C}$ 
with boundary in $\Gamma(\phi)$.
\end{lem}

\begin{proof}
Let $p\in\mathfrak{C}$ be an intersection point 
of $j(\Sigma)$ with $\Gamma(\wh{\phi})$.
There exists an analytic disc $\Delta\subset\Gamma(\wh{\phi})\cap\mathfrak{C}$ 
with boundary in $\Gamma(\phi)\subset\p\mathfrak{C}$
passing through $p$ by Theorem~\ref{Nikolay}(ii). 
As observed after Definition~\ref{disc-w-boundary}, 
$\Delta$ is a codimension one complex submanifold 
of~$\mathfrak{C}$. Since $\mathfrak{C}$ is convex, 
$\Delta=\{f=0\}$ for a function $f\in\OO(\mathfrak{C})$, 
e.g., by~\cite[Theorem 7.4.4]{H}.

By our assumptions, the shifted discs $\Delta_t=s_t(\Delta)=\{f\circ s_{-t}=0\}$,
where $s_t(z,w)=(z,w+it)$, 
do not intersect $j(\Sigma)$ for $t>0$.
Hence, $f_t:=f\circ  s_{-t}\circ j$ is a continuous family of holomorphic functions on $\Sigma$ such that
$f_0$ has a zero and $f_t$, $t>0$, do not. By Rouch\'e's theorem, this is only possible if $f_0\equiv 0$ 
on~$\Sigma$. Hence, $j(\Sigma)\subseteq\Delta$.
\end{proof}

\begin{cor}
\label{common-disc}
Suppose that $\phi_-\le\phi_+$ on $S$. If ${\Gamma(\wh{\phi}_-)}$ and ${\Gamma(\wh{\phi}_+)}$ 
intersect at~$p\in\mathfrak{C}$, then their intersection contains an analytic disc $\Delta\subset\mathfrak{C}$ through $p$ with boundary
in $\Gamma(\phi_-)\cap \Gamma(\phi_+)$.
\end{cor}

\begin{proof}
Let $\Delta=\iota(\mathbb{D})\subset\Gamma(\wh{\phi}_-)$ be an analytic disc through $p$ with boundary in $\Gamma(\phi_-)$
provided by Theorem~\ref{Nikolay}(ii). By Corollary~\ref{hull-monoton}, 
the graph of $\wh{\phi}_-$ lies below the graph of $\wh{\phi}_+$. 
Therefore $\iota:\mathbb{D}\to\mathfrak{C}$
satisfies the assumptions of Lemma~\ref{Rouche} for the function $\phi=\phi_+$ on~$S$.
Hence, $\Delta$ is properly embedded in an analytic disc $\Delta'\subset\mathfrak{C}$ with boundary in $\Gamma(\phi_+)$.
The two discs must coincide and have boundary in $\Gamma(\phi_-)\cap \Gamma(\phi_+)$.
\end{proof}

\section{Hulls and closures of domains}
\label{sec-hulls-closures}

\subsection{Hulls and hyperconvex domains}
\label{subsec-h-and-c-Cn}

The elementary argument in the proof of the next proposition goes
back to~\cite[Proposition 2]{HS}, where the result was
proven for domains with smooth boundaries.

\begin{prop}
\label{boundary-hull-n}
If $\Omega$ is a hyperconvex domain, then
$$
\ph{K}\mathop{\cap} \p\Omega = \phbig{K\mathop{\cap} \p\Omega} \mathop{\cap} \p\Omega
$$
for all compact subsets $K\Subset\conj{\Omega}$.
\end{prop}

\begin{proof}
The inclusion $\supseteq$ is obvious. 
To prove that $\subseteq$ holds, 
suppose that there exists $\phi\in P(\conj{\Omega})$ such that 
$$
m^*:=\max\limits_{\ph{K}\mathop{\cap} \p\Omega} \phi \, > \, \max\limits_{K\mathop{\cap} \p\Omega} \phi
$$
and seek a contradiction.

Consider the compact set
$$
K^*:=\left\{q\in K \mid \phi(q)\ge m^* \right\}.
$$
By our assumption $K^*\mathop{\cap} \p\Omega = \varnothing$. 

Let $\psi\in P(\conj\Omega)$ be a bounded exhaustion function
negative in $\Omega$ and vanishing on~$\p\Omega$. For $C>0$ large enough,
$\wt{\phi}:=\phi + C\psi < m^*$ on $K^*\Subset\Omega$ and therefore everywhere on~$K$. 
On the other hand, $\wt{\phi}=\phi$ on $\p\Omega$ and hence attains the value $m^*$
on $\ph{K}\mathop{\cap} \p\Omega$, which contradicts the definition of the $P(\conj{\Omega})$-hull
because $\wt{\phi}\in P(\conj{\Omega})$.
\end{proof}

\begin{thm}
\label{mainhull-n}
Let $K$ be a compact subset of a bounded pseudoconvex domain~$\Omega$ 
with simple boundary in a Stein manifold.
If its holomorphically or meromorphically convex hull~$\wt{K}$ 
is contained in the closure of $\Omega$, 
then it is contained in $\Omega$.
\end{thm}

\begin{proof}
$\Omega$ has simple boundary and the ambient manifold is Stein, 
so $\wt{K}\subseteq\ph{K}$ by Proposition~\ref{hulls-in-pOhulls}.
Since $\Omega$ is hyperconvex by Proposition~\ref{simple-hyper}, 
we also have $\ph{K}\cap\p\Omega\subseteq \ph{K\cap \p\Omega}$  
by Proposition~\ref{boundary-hull-n}. 
The latter set is empty because $K\Subset\Omega$.
Hence, $\wt{K}\cap\p\Omega=\varnothing$.
\end{proof}

\subsection{Hulls, discs, and boundaries in complex surfaces}
\label{subsec-h-d-b-C2}
In this subsection, we denote by $\wt{K}\supseteq K$ 
any compact `hull' of a compact set $K$ in a complex manifold $M$
satisfying the `local maximum principle' 
\begin{equation}
\label{local-genhull-phull}
\conj{U} \cap  \wt{K} \subseteq \mathop{\mathcal{O}(\conj{U})\text{-hull}}(\p U\cap \wt{K})
\end{equation}
for any open set $U$ with $U\cap K=\varnothing$, where $\mathcal{O}(\conj{U})$ is the algebra
of functions holomorphic on a neighbourhood of~$\conj{U}$.
This inclusion holds for $\wt{K}=\ph{K}$ for {\it any\/} $\Omega$ 
such that $K\subseteq\conj{\Omega}$ by~Theorem~\ref{max-pOhull}. 
If $M$ is Stein, then it also holds for $\wt{K}=\oh{K}$ by~Theorem~\ref{max-phull} and for $\wt{K}=\mh{K}$ by~Theorem~\ref{max-rhull}. 
The following is our main technical result in complex dimension two.

\begin{prop}
\label{maintech}
Let $\Omega$ be a pseudoconvex domain with simple boundary in 
a complex surface and $K\Subset\conj{\Omega}$.
If $\wt{K}\subseteq\conj{\Omega}$, then for every point $p\in X=(\wt{K}-K)\cap\p\Omega$
and every neighbourhood $V\ni p$ there is an analytic disc 
in~$\conj{\Omega}\cap V$ passing through $p$ and with boundary~in~$X$.
\end{prop}

\begin{proof}
Let $(z,w)$ be a coordinate system at $p$ as in Definition~\ref{simple-boundary}
and choose $U$ so small that $\conj{U}\cap K=\varnothing$ and $U\subseteq V$.  
Note that the domain $\Omega_F=\Omega\cap U$ is pseudoconvex.

Let $\phi:\p B\to I=(v_-,v_+)$ be any continuous function  such that
\begin{enumerate}
\item $\phi\le F|_{\p B}$;
\item $\wt{K}\cap (\p B\times I) \subseteq \{(z,w)\in\conj{U} \mid (z,u)\in\p B, v\le \phi(z,u)\}$;
\item $\Gamma(\phi)\cap\Gamma(F|_{\p B})= \wt{K}\cap\Gamma(F|_{\p B})$.
\end{enumerate}
Functions with those properties exist. For instance, one may set
$$
\phi(z,u):= \max \bigl\{F(z,u) - d_{\wt{K}}(z,u,F(z,u)), v_*\bigr\},
$$
where $d_{\wt{K}}$ is the distance function~\eqref{dist-def} 
and $v_*\in (v_-, \min\limits_{\p B} F)$.

Let now $\wh{\phi}:\conj{B}\to\R_v$ be the solution of the Dirichlet problem 
from Theorem~\ref{Nikolay}.
By property (1) and Lemma~\ref{psc-contain} we get 
$$
\wh{\phi}\le F.
$$ 
By property (2) the intersection of $\wt{K}$ with the 
boundary of $U=B\times I$
is contained in the (boundary of the) compact set
$$
K(v_-,\phi):=\{(z,w)\in\C^2\mid (z,u)\in\conj{B}, v\in [v_-,\wh{\phi}(z,u)]\}.
$$
This set is polynomially convex by Corollary~\ref{between-graphs} 
and {\it a fortiori\/} it is $\mathcal{O}(\conj{U})$-convex.
Since $U\cap K=\varnothing$, 
it follows from inclusion~\eqref{local-genhull-phull} that 
$$
U\cap \wt{K} \subseteq K(v_-,\phi).
$$
In other words, $\Gamma(\wh{\phi})$ lies below $\Gamma(F)$ and above $\wt{K}$ in $U=B\times I$.
In~particular, $p\in \Gamma(\wh{\phi})$ because $p\in \wt{K}\cap\Gamma(F)$. 

The function $\psi:= \frac{1}{2}\left(\phi + F|_{\p B}\right)$ has properties (1) and (2)
and hence $p\in \Gamma(\wh{\psi})$ by the same argument. 
Since $\phi\le\psi$ by (1),
we conclude from Corollary~\ref{common-disc} that there is an analytic disc
$\Delta\subset\Gamma(\wh\phi)\cap \Gamma(\wh\psi)\subset\conj{\Omega}$ through $p$ 
with boundary in $\Gamma(\phi)\cap \Gamma(\psi)$.
By property (3) of $\phi$ and the definition of $\psi$, the latter intersection 
is contained in $\wt{K}\cap \Gamma(F)\subseteq (\wt{K}-K)\cap \p\Omega=X$.
\end{proof}

The assertion of the proposition can be strengthened to show that the analytic
discs are in fact contained in~$X$.

\begin{thm}
\label{foliation}
Let $\Omega$ be a pseudoconvex domain with simple boundary
in a complex surface and $K\Subset\conj{\Omega}$ 
a compact subset such that $\wt{K}\subseteq\conj{\Omega}$.
Then the set $X=(\wt{K}-K)\cap\p\Omega$ is a union of analytic discs.
\end{thm}

\begin{proof}
Let $p$ be any point in $X$ and $V\ni p$ a neighbourhood with $\conj{V}\cap K=\varnothing$.
Let us prove that the analytic disc $\Delta \ni p$ constructed in Proposition~\ref{maintech} 
is contained in~$X$. 

First, note that $\Delta$ is contained in $\conj{\Omega}$ and intersects $\p\Omega$. 
Hence, it must be entirely contained in $\p\Omega$ by Proposition~\ref{disc-touching-bd}.

If $\wt{K}=\oh{K}$ is the holomorphically convex hull,
then analytic discs with boundary in $X\subset\oh{K}$ 
are contained in $\oh{K}$ by Lemma~\ref{psh-on-discs},
which completes the proof. A very similar argument 
using Corollary~\ref{psh-on-bd-disc} can be 
given for $\wt{K}=\ph{K}$.

In general, e.g., for meromorphically convex hulls, 
further arguments are needed. Define $E:=\Delta\cap X$. 
This subset is closed in $\Delta$ because $\conj{\Delta}\cap K=\varnothing$.
Note that $p\in E$ and so $E\ne\varnothing$. 
Let us assume that $E\ne\Delta$ and seek a contradiction.

Take a point $\zeta_0\in\Delta-E$
and let $\zeta_1\in E$ be a closest point to $\zeta_0$
in the closed subset $E$ with respect to the hyperbolic 
metric on $\Delta=\iota(\mathbb{D})$.
The geodesic segment $\gamma=[\zeta_0,\zeta_1]$ lies 
in $\Delta-E$ except for the endpoint~$\zeta_1$. 
In particular, if $\Delta'=\iota(\mathbb{D}')$ 
is the hyperbolic disc centred at $\zeta_1$ 
of radius $\dist_\mathbb{D}(\zeta_1,\zeta_0)$,
then every neighbourhood of $\zeta_1$ in $\Delta'$ 
has a boundary point on~$\gamma$ and hence in~$\Delta-E$.

Let $U\ni\zeta_1$ be a neighbourhood as in Definition~\ref{simple-boundary}  
so small that $U\cap \iota(\p\mathbb{D}')=\varnothing$.
Denote by $\Sigma\subset U$ the (properly embedded) 
connected component of $\Delta'\cap U$ containing $\zeta_1$.

By Proposition~\ref{maintech} applied to $\zeta_1\in X$ and $U$,
there exists an analytic disc $\Delta''\subset \conj{\Omega}\cap U$
through $\zeta_1$ with boundary in~$X$. 
As in the proof of Lemma~\ref{Rouche}, we conclude
that $\Delta''\subseteq\Sigma$. 
Namely, $\Sigma=\{f=0\}$ for a holomorphic function $f\in\OO(B\times\R_v)$. 
If $s_t(z,w)=(z,w+it)$, then $s_t(\Sigma)$ 
lies outside of $\conj{\Omega}\cap U$ for~$t>0$
because $\Sigma\subset\p\Omega$.
Thus $f$ has a zero on $\Delta''$ 
and its shifts $f\circ s_{-t}$, $t>0$, do not.
By Rouch\'e's theorem, 
$f$ must vanish on $\Delta''$ identically.
 
Hence, $\Delta''\subseteq\Delta'$ defines a neighbourhood
of $\zeta_1$ in $\Delta'$ with boundary entirely
contained in $E=\Delta\cap X$. This contradicts
our choice of $\zeta_1$ and $\Delta'$, which proves
the theorem.
\end{proof}

\begin{rem}
\label{Catlin}
Catlin~\cite[Proposition 3.1.12]{Ca} proved the theorem 
for a smoothly bounded $\Omega\subset\C^2$ 
and $\wt{K}=\ph{K}$ by a completely different method. 
He also gave an example~\cite[Theorem 3.1.14]{Ca}
showing that the result does not extend to $\C^n$ for~$n\ge 3$.
The idea is that using \cite[Theorem 1.3.8]{Sto} 
one can construct a smoothly bounded pseudoconvex domain in $\C^{n+1}$
such that its boundary contains any given polynomially
convex set $L\Subset\C^n\subset\C^{n+1}$ and is
strictly pseudoconvex on the complement of~$L$. 
Then one may take $L=\wh{K}$ for a set $K$ 
in $\C^n$ without analytic discs 
in the added hull $\wh{K}-K$, see~\cite{St1} or~\cite[\S 1.7]{Sto}.
\end{rem}

The results of this subsection can be used to give
alternative proofs of Proposition~\ref{boundary-hull-n} and 
Theorem~\ref{mainhull-n} for (not necessarily bounded) domains 
with simple boundary in Stein complex surfaces.

\begin{prop}
\label{boundary-hull}
If $\Omega$ is a pseudoconvex domain with simple boundary 
in a Stein complex surface $M$,
then
$$
\ph{K}\mathop{\cap} \p\Omega = \phbig{K\mathop{\cap} \p\Omega} \mathop{\cap} \p\Omega
$$
for all compact subsets $K\Subset\conj{\Omega}$.
\end{prop}

\begin{proof}
The inclusion $\supseteq$ is obvious. The proof of $\subseteq$ 
is by contradiction.
Suppose that there exists $\phi\in P(\conj{\Omega})$ such that 
$$
m^*:=\max\limits_{\ph{K}\mathop{\cap} \p\Omega} \phi \, > \, \max\limits_{K\mathop{\cap} \p\Omega} \phi
$$
and let
$$
E:= \{ q\in \ph{K}\mathop{\cap} \p\Omega \mid \phi(q)=m^*\}.
$$
By Theorem~\ref{foliation} for $\wt{K}=\ph{K}$, there is an analytic disc 
in $(\ph{K}-K)\mathop{\cap}\p\Omega$ through every point of $E$. 
The function $\phi$ is sh on this disc by Corollary~\ref{psh-on-bd-disc}
and attains its maximum equal to $m^*$ there. Hence, $\phi$
is identically equal to $m^*$ on the disc and so the disc is 
contained in $E$ by the definition of $E$.

Thus, $E$ is a compact subset of $M$ that is a union of analytic discs.
Such a subset in the Stein manifold $M$ must be empty. 
For otherwise a {\it strictly\/} psh function $\Phi$ on $M$ 
will attain its maximum on~$E$ at a point $p_0\in E$ 
and hence on a disc through that point, 
which is impossible.
\end{proof}

\begin{thm}
\label{mainhull}
Let $K$ be a compact subset of a pseudoconvex domain~$\Omega$ 
with simple boundary in a Stein complex surface.
If its holomorphically or meromorphically convex hull
is contained in the closure of $\Omega$, 
then it is contained in $\Omega$.
\end{thm}

\begin{proof}
The hull is contained in $\ph{K}$ by Proposition~\ref{hulls-in-pOhulls},
and the result follows from Proposition~\ref{boundary-hull}
because $K\cap\p\Omega=\varnothing$.
\end{proof}

\section{Strictly pseudoconvex approximations}

\subsection{Strictly pseudoconvex approximations to $C^1$ boundaries}
\label{subsec-spsc}
For the proof of the main result of this subsection
we shall need several standard facts from analysis.
The first is a version of the tubular neighbourhood theorem
without loss of regularity.

\begin{lem}
\label{normal_nbh}
Let $H$ be a co-oriented $C^k$ hypersurface, $1\le k\le\infty$,
in a smooth manifold~$M$. 
Then a neighbourhood of $H$ is $C^k$ diffeomorphic 
to a neighbourhood of $H\times \{0\}\subset H\times\R_t$ 
co-oriented by~$dt>0$.
\end{lem}

\begin{proof}[Sketch of proof]
Using a partition of unity, 
we can construct a $C^k$ vector field $X:U\to TM$
on a neighbourhood $U\supset H$ that is positive 
with respect to the co-orientation on~$H$.
By the Picard--Lindel\"of theorem 
and the implicit function theorem, 
there exists a neighbourhood $V$ of $H\times \{0\}$ in $H\times\R$
such that the $C^k$ map $V\ni (p,t)\mapsto e^{tX}(p)\in M$
is the required  diffeomorphism onto a neighbourhood of $H$ 
in~$M$.
\end{proof}

Next we recall a key fact about smoothing Lipschitz functions
by convolution. Let $f:U\to\R$ be a locally Lipschitz function on 
an open subset of $\R^m$. By Rademacher's theorem, 
the gradient $\nabla f$ is defined a.\,e.\ in $U$
and locally bounded. Let $\chi:\R^m\to\R_{\ge 0}$
be a `mollifier', i.e.\ a smooth non-negative 
compactly supported function
such that $\int_{\R^m} \chi\, d\lambda_m =1$,
where $d\lambda_m$ is the Lebesgue measure.
Then 
$$
f\star \chi (x) = \int\limits_{\R^m} f(x-y)\chi(y)\, d\lambda_m(y)
$$
is a smooth function on $\{x\in U\mid x-y\in U \text{ for all } y\in\supp\chi\}$
and 
\begin{equation}
\label{grad_lip_conv}
\nabla\bigl(f\star \chi\bigr) (x) = \int\limits_{\R^m} \nabla f(x-y)\chi(y)\, d\lambda_m(y).
\end{equation}
For a proof, see~\cite[Proof of Theorem 9.67]{RW}.
 
\begin{exm}[Regularised maximum function]
\label{reg_max}
Let $\theta:\R\to\R_{\ge 0}$ be an even mollifier 
supported in $[-1,1]$. For any $\eta=(\eta_1,\ldots,\eta_m)\in\R^m_{>0}$,
the regularised maximum function on $\R^m$ is defined by
$$
M_\eta (x) = \int\limits_{\R^m} \max\{x_1+s_1,\ldots,x_m+s_m\} \prod\limits_{j=1}^m \eta_j^{-1}\theta(s_j/\eta_j)\, d\lambda_m(s),
$$
see \cite[Lemma I.5.18]{D2}.
The function $\max\{x_1,\ldots,x_m\}$ is Lipschitz on $\R^m$ 
and its gradient is equal to $e_j=\nabla x_j$
on the set where $x_j>x_\ell$ for all $\ell\ne j$. 
The union of these sets is of full Lebesgue measure in $\R^m$. 
Applying \eqref{grad_lip_conv}, we conclude that the gradient of $M_\eta$
lies in the convex hull of $e_j=\nabla x_j$. 
Hence, if $\phi_1,...,\phi_m$ are smooth functions on an open subset $V\subset\C^n$, 
then 
\begin{equation}
\label{reg_max_grad}
\nabla \bigl(M_\eta(\phi_1,\ldots,\phi_m)\bigr)(z) 
\in \mathop{\mathrm{conv}} \bigl\{\nabla \phi_1(z),\ldots,\nabla \phi_m(z)\bigr\}
\end{equation} 
for all $z\in V$ by the chain rule.\qed
\end{exm}

Finally, we need several known properties of the (Euclidean) distance functions.
For a closed subset $E\subset\R^m$, the distance function
\begin{equation}
\label{dist-def}
d_E(x):=\inf\{|x-y| \mid y\in E\}
\end{equation}
is Lipschitz (with constant~$1$). If $d_E$ is
differentiable at $x\in\R^m-E$, then 
\begin{equation}
\label{dist-grad}
\nabla d_E(x)=\frac{x-p_E(x)}{|x-p_E(x)|},
\end{equation}
where $p_E(x)\in E$ is a point such that $d_E(x)=|x-p_E(x)|$,
which is therefore necessarily unique. 
(To obtain~\eqref{dist-grad}, observe first that
$d_E\bigl(x+t(p_E(x)-x)\bigr)=(1-t)|x-p_E(x)|$ for all $t\in [0,1]$.
Differentiating at~$t=0$ gives $\nabla d_E(x)\boldsymbol{\cdot} (x-p_E(x))=|x-p_E(x)|$,
where $\boldsymbol{\cdot}$ denotes the Euclidean scalar product.
Since $|\nabla d_E(x)|\le 1$ by the Lipschitz bound, 
formula~\eqref{dist-grad} follows from the equality case of Cauchy--Schwarz.)

Let now $H=\{\rho=0\}$, $d\rho\ne 0$ on $H$, be a closed hypersurface of class $C^k$, $k\ge 1$, in~$\R^m$.
It follows from~\eqref{dist-grad} and Fermat's principle that if $d_H$ is differentiable at $x\notin H$,
then 
\begin{equation}
\label{dist-grad-hyp}
\nabla d_H(x)=\nu_H(p_H(x))
\end{equation}
is the unit normal to $H$ at the (unique) nearest point $p_H(x)\in H$ 
pointing in the direction of~$x$.
The next lemma shows that $d_H$ also has a `one-sided' gradient equal 
to a unit normal to $H$ at each point in~$H$.

\begin{lem}
\label{dist-grad-on-hyp}
Let $x_0\in H$, then
$$
d_H(x)=\nu^+_H(x_0)\boldsymbol{\cdot} (x-x_0) + o(|x-x_0|) \text{ as } x\to x_0 \text{ with } \rho(x)\ge 0,
$$
where $d\rho\bigl(\nu^+_H(x_0)\bigr)>0$ and $\boldsymbol{\cdot}$ is the Euclidean scalar product.
\end{lem}

\begin{rem}
\label{signed-dist-grad-on-hyp}
Let $d^{\,\uparrow}_H:=(\mathop{\mathrm{sign}}\rho)\, d_H$ be the {\it signed\/} distance function.
The lemma asserts that $\nabla d^{\,\uparrow}_H(x_0)=\nu^+_H(x_0)$ for all $x_0\in H$.
\end{rem}
 
\begin{proof}[Proof of the lemma]
Without loss of generality, we may assume $x_0=0$. Let $x'\in H$ be any point such
that $|x-x'|=d_H(x)$. Note that 
$$|x'|\le |x|+d_H(x)\le 2|x|
$$ 
by the triangle inequality and the definition of $d_H$.
Furthermore, 
$$
x=x'+d_H(x)\nu^+_H(x')
$$ 
because $x-x'$ is orthogonal to~$T_{x'}H$ by Fermat's principle. Hence,
$$
\nu^+_H(0)\boldsymbol{\cdot} x = \nu^+_H(0)\boldsymbol{\cdot} x' + d_H(x)\, \nu^+_H(0)\boldsymbol{\cdot}\nu^+_H(x').
$$ 
It remains to observe that for a $C^1$-hypersurface,
$\nu^+_H(0)\boldsymbol{\cdot} x'=o(|x'|)$ 
and $\nu^+_H(0)\boldsymbol{\cdot}\nu^+_H(x')=1+o(1)$ as $x'\to 0$
with $x'\in H$.
\end{proof}

If the hypersurface $H$ is of class $C^k$ with $k\ge 2$, there exists
a neighbourhood $V\supset H$ such that $d^{\,\uparrow}_H\in C^k(V)$, see, e.g.,~\cite[p.~155]{Fo}.

If $H$ is only $C^1$, then $d_H$ may fail to be differentiable on
a dense subset of $\R^m-H$, see~\cite{Sa}. Note, however, that if $x_0\in H$ and
$d_H$ is differentiable at $x\in \R^m-H$, then it follows from 
\eqref{dist-grad-hyp} and the triangle inequality that
\begin{equation}
\label{C1-on-hyp}
\bigl|\nabla d^{\,\uparrow}_H(x) - \nu_H^+(x_0)\bigr| 
\le \omega_{\nu_H}(x_0,2|x-x_0|),
\end{equation}
where $\omega_{\nu_H}$ is the modulus of continuity of the unit normals $\nu^\pm_H$.

\begin{prop}
\label{def-func}
Let $\Omega$ be a bounded pseudoconvex domain with $C^1$ boundary
in a Stein manifold. 
Then $\Omega$ has a $C^1$ defining function $\phi$
on a neighbourhood of $\conj\Omega$ such that $-\log(-\phi)$ 
is a smooth strictly psh function on~$\Omega$.
\end{prop}

\begin{proof}
1) Let us consider an {\it arbitrary\/} pseudoconvex domain $\Omega\subsetneq\C^n$.
We will construct a particular regularisation of the function 
$$\psi(z):=-e^{-|z|^2}d_{\p\Omega}(z), \quad z\in \conj{\Omega},
$$ 
by applying Richberg's approximation as in~\cite[Theorem~I.5.21]{D2} to 
the continuous strictly psh function
$$
\wt{\psi}(z):= -\log(-\psi(z)) = -\log d_{\p\Omega}(z) + |z|^2, \quad z\in\Omega.
$$
(Note that $-\log d_{\p\Omega}$ is psh in the pseudoconvex domain $\Omega$ by~\cite[\S 2.6]{H}.)

Choose $z_\alpha\in\Omega$ 
and $0<r'_\alpha<r''_\alpha< d_{\p\Omega}(z_\alpha)$
so that the open balls $B'_\alpha=B(z_\alpha,r'_\alpha)$ cover $\Omega$
and the covering of $\Omega$ by the closed balls
$\conj{B''_\alpha}=\conj{B(z_\alpha,r''_\alpha)}$
is locally finite. 

For every $\alpha$, let
\begin{equation}
\label{d-choice}
d_\alpha: = \min\limits_{z\in \conj{B''_\alpha}} d_{\p\Omega}(z)>0.
\end{equation}
Define the intermediate radius 
$r_\alpha\in (r'_\alpha, r''_\alpha)$ by the condition
\begin{equation}
\label{r-choice}
\frac{{r''_\alpha}^2 - {r_\alpha}^2}{2} =
{r_\alpha}^2 - {r'_\alpha}^2. 
\end{equation}
Choose $\delta_\alpha\in (0,1)$ so that
\begin{equation}
\label{delta-choice}
\delta_\alpha {r_\alpha}^2 < \frac{1}{3}d_\alpha 
\end{equation}
and set
\begin{equation}
\label{eta-choice}
\eta_\alpha:= \delta_\alpha \frac{{r''_\alpha}^2 - {r_\alpha}^2}{2}
= \delta_\alpha ({r_\alpha}^2 - {r'_\alpha}^2)
< \frac{1}{3}d_\alpha .
\end{equation}
Next choose $\eps_\alpha\in (0,d_\alpha)$ so that
\begin{equation}
\label{eps-choice}
\bigl|\wt{\psi}(z')-\wt{\psi}(z)\bigr| < \eta_\alpha
\end{equation}
for all $z\in\conj{B''_\alpha}$ and $|z'-z|<\eps_\alpha$.

Let $\chi_\eps(z)=\eps^{-2n}\chi(\frac{z}{\eps})$ for $\eps>0$ 
and a spherically symmetric mollifier $\chi$ supported 
in the unit ball in $\C^n$. 
By standard properties of psh functions~\cite[\S I.5.A]{D2}
or~\cite[\S 2.6]{H}, the convolution
$\wt{\psi}\star \chi_{\eps_\alpha}$ is $\ge\wt{\psi}$ 
on $\conj{B''_\alpha}$ 
and its Levi form is bounded below by the Levi form of~$|z|^2$. 
Therefore, 
\begin{equation}
\label{phi-alpha}
\wt{\phi}_\alpha(z):= \wt{\psi}\star \chi_{\eps_\alpha}(z) + \delta_\alpha ({r_\alpha}^2 - |z-z_\alpha|^2)
\end{equation} 
is a smooth strictly psh function on $\conj{B''_\alpha}$ such that 
$$
\wt{\phi}_\alpha(z) \ge \wt{\psi}(z) + \delta_\alpha ({r_\alpha}^2 - |z-z_\alpha|^2).
$$
By~\eqref{eps-choice} and elementary properties of the convolution, 
we also get
$$
\wt{\phi}_\alpha(z) \le \wt{\psi}(z) +\eta_\alpha + 
\delta_\alpha ({r_\alpha}^2 - |z-z_\alpha|^2)
\text{ for }z\in \conj{B''_\alpha}.
$$
Hence, using~\eqref{eta-choice} and~\eqref{delta-choice}, we see that
\begin{align}
\wt{\phi}_\alpha > \wt{\psi}+\eta_\alpha     & \;\text{ on }   B'_\alpha, \label{passing-eta1} \\ 
\wt{\phi}_\alpha \le \wt{\psi} - \eta_\alpha & \;\text{ on } \p B''_\alpha, \label{passing-eta2} \\
\wt{\phi}_\alpha +\eta_\alpha \le \wt{\psi} + d_\alpha & \;\text{ on } B''_\alpha. \label{upper-bound-alpha}
\end{align}

For $z\in\Omega$, let 
$\{\alpha_{1\vphantom{)}},\ldots,\alpha_{m(z)}\}=\{\alpha\mid z\in B''_\alpha\,\}
$
and define 
$$
\wt{\phi}(z):=
M_{(\eta_{\alpha_{1\vphantom{)}}},\ldots,\,\eta_{\alpha_{m(z)}})}\bigl(\wt{\phi}_{\alpha_{1\vphantom{)}}}(z),\ldots,\wt{\phi}_{\alpha_{m(z)}}(z)\bigr)
$$
with $M_\eta$ the regularised maximum function on $\R^{m(z)}$ from Example~\ref{reg_max}.
Estimates \eqref{passing-eta1} and~\eqref{upper-bound-alpha}
imply that
\begin{equation}
\label{bound-phi-tilde}
\wt{\psi}(z) \le \wt{\phi}(z) \le \wt{\psi}(z) + d_{\p\Omega}(z).
\end{equation}
From \eqref{passing-eta1} and \eqref{passing-eta2}, we infer that
\begin{equation}
\label{local-phi-tilde}
\wt{\phi}(z')=
M_{(\eta_{\alpha_{1\vphantom{)}}},\ldots,\,\eta_{\alpha_{m(z)}})}\bigl(\wt{\phi}_{\alpha_{1\vphantom{)}}}(z'),\ldots,\wt{\phi}_{\alpha_{m(z)}}(z')\bigr)
\end{equation}
for all $z'$ in a sufficiently small open neighbourhood of $z$. 
Consequently, $\wt{\phi}$ is a smooth strictly psh function on $\Omega$
by~\cite[Lemma~I.5.18(e)]{D2}.

We may now define a function $\phi\in C^\infty(\Omega)\cap C^{0}(\conj{\Omega})$ by setting
$$
\phi:=\left\{
\begin{array}{ll}
-\exp(-\wt{\phi}) & \text{ on } \Omega,\\
0 & \text{ on } \p\Omega.
\end{array}
\right.
$$
Inequalities~\eqref{bound-phi-tilde} show that
\begin{equation}
\label{phi-at-boundary}
\begin{array}{rcl}
\phi(z) & = & \psi(z)\,\bigl(1+O(d_{\p\Omega}(z))\bigr)\\[2pt]
& = & -e^{-|z|^2}d_{\p\Omega}(z) + o(d_{\p\Omega}(z)) \;\text{ as } z\to\p\Omega.
\end{array}
\end{equation}

\smallskip
\noindent
2) Suppose next that $\Omega\subsetneq\C^n$ has $C^1$ boundary at $z_0\in\p\Omega$. 
Then $d_{\p\Omega}$ coincides with the
distance to a $C^1$ hypersurface in a neighbourhood of $z_0$. 
In particular, it follows from Lemma~\ref{dist-grad-on-hyp} and~\eqref{phi-at-boundary}
that $\phi$ has a one-sided gradient 
\begin{equation}
\label{grad-phi-z0}
\nabla_{\conj{\Omega}}\, \phi (z_0) = e^{-|z_0|^2}\nu_{\p\Omega}(z_0),
\end{equation}
where $\nu_{\p\Omega}$ is the outer unit normal to $\p\Omega$.

Let us show that 
$$
\nabla\phi(z) \to e^{-|z_0|^2}\nu_{\p\Omega}(z_0) 
\;\text{ as } \Omega\ni z\to z_0\in\p\Omega.
$$
From formulas \eqref{local-phi-tilde} and~\eqref{reg_max_grad}, we obtain 
\begin{equation}
\label{grad-phi-in-conv}
\nabla\phi(z)=-\phi(z)\nabla \wt{\phi}(z) 
\in \mathop{\mathrm{conv}} \bigl\{-\phi(z)\nabla \wt{\phi}_\alpha(z)\bigr\}_{\{\alpha\,\mid\, z\in B''_\alpha\,\}}.
\end{equation}
Applying~\eqref{grad_lip_conv} to~\eqref{phi-alpha}
shows that
\begin{equation}
\label{grad_phi_alpha}
\nabla \wt{\phi}_\alpha(z) = -\frac{\nabla\psi}{\psi}\star \chi_{\eps_\alpha}(z) - \delta_\alpha \nabla |z-z_\alpha|^2.
\end{equation}
The norm of the last term is bounded by
$$
2\delta_\alpha|z-z_\alpha|
\le 2\delta_\alpha r''_\alpha 
\overset{\eqref{r-choice}}{<} 2\sqrt{3}\delta_\alpha r_\alpha 
\underset{\delta_\alpha<1}{\overset{\eqref{delta-choice}}{<}} 2\sqrt{d_\alpha}
\overset{\eqref{d-choice}}{\le} 2\sqrt{d_{\p\Omega}(z)}.
$$
The integral in the convolution is taken over the ball centred at $z$ 
of radius $\eps_\alpha<d_\alpha\le d_{\p\Omega}(z)$.
For $z'$ in this ball, $|z'-z_0|\le 2|z-z_0|$.
Hence, once $z$ is close enough to $z_0$, it follows from~\eqref{C1-on-hyp} that
$$
\begin{array}{rcl}
\bigl|\nabla\psi(z') - e^{-|z_0|^2}\nu_{\p\Omega}(z_0)\bigr| &
\le &
e^{-|z_0|^2}\bigl|-\nabla d_{\p\Omega}(z')-\nu_{\p\Omega}(z_0)\bigr| \\[2pt]
&& + \bigl|e^{-|z_0|^2}-e^{-|z'|^2}\bigr|\cdot\bigl|\nabla d_{\p\Omega}(z')\bigr| \\[2pt]
&&+ d_{\p\Omega}(z')\bigl|\nabla e^{-|z'|^2}\bigr| \\[2pt]
&{\le} &
\omega_{\nu_{\p\Omega}}(z_0,4|z-z_0|) + {\mathrm{const}}\cdot|z-z_0|
\end{array}
$$
for all $z'\in B(z,\eps_\alpha)$ at which $d_{\p\Omega}$ is differentiable,
which implies a.\,e.\ by Rademacher's theorem. (The constant is absolute
since $|\nabla d_{\p\Omega}(z')|=1$ and 
$\bigl|\nabla e^{-|\zeta|^2}\bigr|\le \sqrt{2/e}$ for all $\zeta\in\C^n$.)
Furthermore, 
$$
\psi(z')= \psi(z)\bigl(1+O(d_{\p\Omega}(z))\bigr)
$$
by~\eqref{eps-choice} and~\eqref{eta-choice}.
Recalling also that $\phi(z)=\psi(z)\bigl(1+O(d_{\p\Omega}(z))\bigr)$ by~\eqref{phi-at-boundary}, 
we conclude from~\eqref{grad_phi_alpha} that 
$$
\left|-\phi(z)\nabla \wt{\phi}_\alpha(z) - e^{-|z_0|^2}\nu_{\p\Omega}(z_0)\right|
$$
is bounded from above by a function of $z$ tending to zero as $z\to z_0$
independent of $\alpha$. Thus, our claim about the continuity of
the gradient of $\phi$ at $z_0$ follows from~\eqref{grad-phi-in-conv}.

\smallskip
\noindent
3) Let $S\subseteq \p\Omega$ be the (relatively open) subset of points at which
$\Omega\subsetneq\C^n$ has $C^1$ boundary. 
By Lemma~\ref{normal_nbh}, a neighbourhood of $S$ 
is $C^1$ diffeomorphic to a neighbourhood of $S\times\{0\}$ in $S\times\R_t$
so that $t<0$ inside $\Omega$.
Define $\phi(z_0,t):=-\phi(z_0,-t)$ for $z_0\in S$
and $t>0$ small enough (depending on $z_0$). 
This extends $\phi$ to a $C^1$ function on a neighbourhood of $\Omega\sqcup S$ 
with the gradient at $z_0\in S$ given by~\eqref{grad-phi-z0}, 
which completes the proof of the proposition for bounded pseudoconvex domains 
with $C^1$ boundary in $\C^n$.

\smallskip
\noindent
4) If $\Omega$ is a domain in a Stein manifold $M$, 
we follow part 1) of the proof of \cite[Lemma 1]{DF}. 
Embed $M$ in~$\C^n$ 
and take a Stein neighbourhood $W\supset M$ 
with a holomorphic retraction $\pi:W\to M$, see~\cite[\S 3.3]{For}.
The pre-image $\check{\Omega}=\pi^{-1}(\Omega)$ 
is a pseudoconvex domain in $\C^n$. 
Clearly, $\check{\Omega}$ has $C^1$ boundary at every point 
in $\p\check{\Omega}\cap M=\p_M\Omega$
and $\p\check{\Omega}$ is transverse to $M$. 
By part 3), there exists a $C^1$ function $\phi$ on a neighbourhood 
of $\check{\Omega}\sqcup (\p\check{\Omega}\cap M)$ 
vanishing with non-zero gradient on~$\p\check{\Omega}$ near $M$
and such that $-\log(-\phi)$ is smooth and strictly psh in $\check{\Omega}$.
The~restriction of $\phi$ to $M$ is the required 
defining function for~$\Omega$.
\end{proof}

\begin{rem}
\label{def-func-connected}
The set of defining functions satisfying the conclusions of Proposition~\ref{def-func}
for a given $\Omega$ is connected. Indeed, suppose that $\phi$ and $\varpi$ are two such functions.
Then $\varpi=e^f\phi$ on a neighbourhood $V\supset\p\Omega$ for a $C^1$ function~$f$.
For $s\in[0,1]$, define
$$
\phi_s:=\left\{
\begin{array}{ll}
-(-\varpi)^s(-\phi)^{1-s} & \text{ on } \Omega,\\
e^{sf}\phi & \text{ on } V.
\end{array}
\right.
$$
Then $\phi_0=\phi$, $\phi_1=\varpi$, 
and $\phi_s$ is a $C^1$ defining function for $\Omega$ 
such that $-\log(-\phi_s)$ is smooth and strictly psh in~$\Omega$
for every $s\in [0,1]$.
\end{rem}

\begin{rem}
\label{Harrington}
Since $C^1$ boundaries are Lipschitz, \cite[Theorem 1.1]{Ha}
provides a Lipschitz defining function $\phi$ for $\Omega$ such that
$-(-\phi)^\eta$ is strictly psh in $\Omega$ for some $\eta>0$. 
One may ask whether the proof of Proposition~\ref{def-func} could be combined 
with the methods of~\cite{Ha} 
to obtain an exact analogue of~\cite{DF} for pseudoconvex domains with $C^1$ boundary.
\end{rem}

\subsection{Holomorphic and meromorphic convexity}
Our next goal is to prove that the strictly pseudoconvex deformations
constructed in Proposition~\ref{def-func} inherit the convexity properties
of the domain.

\begin{lem}
\label{Stein-boundary}
Let $\Omega$ be a domain with $C^1$ boundary in a Stein manifold
such that its closure $\conj{\Omega}$ 
is holomorphically or meromorphically convex.
If $K\Subset\Omega$, then the holomorphically or, respectively,
meromorphically convex hull of~$K$ is contained in~$\Omega$.
\end{lem}

\begin{proof}
Every holomorphically or meromorphically convex subset is 
an intersection of complements of complex hypersurfaces
and therefore its interior is locally pseudoconvex 
by~\cite[Corollary 2.5.7]{H}. 
Hence, $\Omega$ is a pseudoconvex bounded domain with simple boundary, cf. Remark~\ref{sb-properties}. 
The assertion follows now from Theorem~\ref{mainhull-n}.
\end{proof} 

\begin{rem}[A shortcut in $\C^n$]
\label{izzo-boundary}
A short direct proof of Lemma~\ref{Stein-boundary}
for domains in $\C^n$ was communicated to us by Alexander Izzo.
Arguing by contradiction, suppose that $z_0\in \wt{K}\cap\p\Omega\ne\varnothing$,
where $\wt{K}$ denotes the holomorphically or meromorphically
convex hull of~$K$. 
If $\tau:\C^n\to \C^n$ is a small translation 
in the direction of the outer normal $\nu_{\p\Omega}(z_0)$,
then $\tau(K)\Subset\Omega$ but 
$\wt{\tau(K)}=\tau(\wt{K})\not\subset\conj{\Omega}=\wt{\conj{\Omega}}$,
which is clearly impossible.
\end{rem}

\begin{cor}
\label{convex-closures-n}
Let $\Omega$ be a domain with $C^1$ boundary in a Stein manifold
such that its closure is holomorphically or meromorphically convex.
If $\psi$ is a psh function on $\Omega$ and $K_\psi:=\{\psi\le 0\}\Subset\Omega$,
then this compact set is 
holomorphically or, respectively, meromorphically convex.
\end{cor}

\begin{proof}
The holomorphically or meromorphically convex hull 
$\wt{K}_\psi$ of $K_\psi$ 
is contained in $\Omega$ by Lemma~\ref{Stein-boundary}. 
Take an open set $\Omega'\supset\wt{K}_\psi$ 
so that $\conj{\Omega'}\subset\Omega$ and let $U:=\Omega'-K_\psi$.
If $\wt{K}_\psi\ne K_\psi$, then $\psi$
is positive at some point $p\in U \cap \wt{K}_\psi$.
This contradicts Theorems~\ref{max-phull} and~\ref{max-rhull}
because $\p U\cap \wt{K}_\psi \subseteq K_\psi$.
\end{proof}

\begin{thm}
\label{spsc-approx}
Let $\Omega$ be a domain with $C^1$ boundary 
in a Stein manifold such that its closure 
is holomorphically or meromorphically convex. 
Then $\Omega$ is isotopic 
{\rm (}by an arbitrarily small ambient $C^1$ isotopy\/{\rm )} 
to a smoothly bounded strictly pseudoconvex domain 
$\Omega'$ with holomorphically or, respectively, 
meromorphically convex closure $\conj{\Omega}{}'\Subset\Omega$.
\end{thm}

\begin{proof}
As shown in the proof of Lemma~\ref{Stein-boundary} above,
$\Omega$ is pseudoconvex. 
Let $\phi$ be a $C^1$ defining function for $\Omega$ such that $-\log(-\phi)$ 
is smooth and strictly psh in $\Omega$ provided by Proposition~\ref{def-func}.
For $\eps>0$ small enough, 
the family of real hypersurfaces $\{\phi= -s\}$ with $s\in [0,\eps]$
is a $C^1$~isotopy by Lemma~\ref{normal_nbh}
and the implicit function theorem. Hence, the smoothly bounded
strictly pseudoconvex domain $\Omega':=\{\phi< -\eps\}$ 
is ambiently $C^1$ isotopic to $\Omega$ by the isotopy extension theorem
(see \cite{B} for a detailed proof of the latter in the $C^1$ case).

$\conj{\Omega}{}'=\{-\log(-\phi)\le -\log\eps\}\Subset\Omega$
is holomorphically or, respectively, meromorphically convex by Corollary~\ref{convex-closures-n}
with $\psi=-\log(-\phi/\eps)$.
\end{proof}

Theorem~\ref{spsc-approx-cn} is an immediate corollary of
Theorem~\ref{spsc-approx} and equalities \eqref{pol=hol} and~\eqref{rat=mer}. 
Note also that by using Remark~\ref{izzo-boundary},
we obtain a proof of Theorem~\ref{spsc-approx-cn} 
that does not depend on Sections~\ref{sec-simple-boundary}--\ref{sec-hulls-closures}.

\begin{rem}
\label{spsc-connected}
It follows from Remark~\ref{def-func-connected} that 
different choices of the defining function $\phi$ 
in the proof of Theorem~\ref{spsc-approx} 
produce domains~$\Omega'$ that are smoothly isotopic 
through strictly pseudoconvex subdomains in $\Omega$ 
with holomorphically/meromorphically convex closure. 
\end{rem}

\begin{rem}
\label{c-infty}
If $\Omega$ has $C^k$ boundary with $2\le k\le \infty$, the proof 
of Theorem~\ref{spsc-approx} can be considerably shortened. 
By \cite[Theorem~1]{DF} there exists a $C^k$ defining 
function $\rho$ on a neighbourhood $U\supset\conj{\Omega}$ 
such that, for $\eta>0$ sufficiently small, 
$\wt{\rho}=-(-\rho)^\eta$ is strictly psh in~$\Omega$. 
The latter is an open condition on $\rho$ in the 
Whitney topology on $C^k(\Omega)$ for $k\ge 2$.
Since smooth functions are dense in the Whitney topology on $C^k(\Omega)$, 
there is a smooth function $\sigma<0$ on $\Omega$ such that
$-(-\sigma)^\eta$ is strictly psh and 
$\bigl|\mathcal{D}\sigma(z) -\mathcal{D}\rho(z)\bigr|\to 0$ 
as $z\to\p\Omega$ for all (real) higher-order partial derivatives $\mathcal{D}$ 
of order~$\le k$.
Hence, if we define 
$$
\phi:=\left\{
\begin{array}{ll}
\sigma & \text{ on } \Omega,\\
\rho & \text{ on } U-\Omega,
\end{array}
\right.
$$
then $\phi$ is a $C^k$ defining function for $\Omega$ 
such that already $-(-\phi)^\eta$ 
is smooth and strictly psh in $\Omega$.
Now set $\Omega'=\{\phi<-\eps\}$ for small enough $\eps>0$
and apply Corollary~\ref{convex-closures-n}.
\end{rem}

\subsection{Applications to the geometric topology of polynomially and rationally convex sets in $\C^2$}
Let us first recall from the work of Cieliebak and Eliashberg~\cite{CE2}
that if $n\ge 3$, then every smoothly bounded strictly pseudoconvex domain 
in $\C^n$ is isotopic to a strictly pseudoconvex domain with rationally
convex closure. 
If one assumes additionally that the $n$-th homology group 
of the domain vanishes for all constant coefficients, 
then it is isotopic to a strictly pseudoconvex domain 
with polynomially convex closure. 

We can now show that such statements fail in $\C^2$ even if one
allows deformations to domains with merely $C^1$ boundaries. 

\begin{thm}
\label{NS-MT-gen}
There exist bounded strictly pseudoconvex domains in $\C^2$ 
such that their closures are not $C^1$ diffeomorphic 
to rationally convex compact subsets in $\C^2$. 

There exist contractible bounded strictly pseudoconvex 
domains in~$\C^2$ 
such that their closures are not $C^1$ diffeomorphic
to polynomially convex compact subsets in~$\C^2$.
\end{thm}
 
\begin{proof}
The cotangent unit disc bundle of the Klein bottle 
and the disc bundle over $\R\mathrm{P}^2$
with orientable total space and Euler number~$-2$ 
are diffeomorphic to closures 
of strictly pseudoconvex domains in~$\C^2$, see~\cite[\S 10.5]{For},
but they cannot be diffeomorphic to rationally convex closures
of strictly pseudoconvex domains in~$\C^2$ by~\cite[Theorem 1.6]{NS}. 
(The results in~\cite{NS} concern 
orientation preserving diffeomorphisms, see \cite[Remark 1.5]{NS}. 
However, the cotangent disc bundle of the Klein
bottle has an orientation reversing diffeomorphism and 
the open disc bundle over $\R\mathrm{P}^2$ with orientable total space
and Euler number $+2$ is not orientably diffeomorphic 
to a Stein domain at all~\cite[\S 5]{N}.) 

If one of these two manifolds were $C^1$ diffeomorphic 
to a rationally convex subset in~$\C^2$, then it would also be 
$C^1$ diffeomorphic to the rationally convex closure
of a smoothly bounded strictly pseudoconvex domain in~$\C^2$ 
by Theorem~\ref{spsc-approx-cn}.
Since $C^1$ diffeomorphic smooth manifolds are diffeomorphic 
by a classical result of Whitney, we arrive at 
a contradiction.

For the second part, consider the contractible $4$-manifolds 
admitting strictly pseudoconvex embeddings in $\C^2$
from~\cite[Example~3.2]{G}. These manifolds are not
diffeomorphic to polynomially convex closures of 
strictly pseudoconvex domains because then they
would have to be diffeomorphic to the standard $4$-ball 
by~\cite[Theorem~1.7]{MT}. 
Applying Theorem~\ref{spsc-approx-cn} completes the proof.
\end{proof}

\appendix
\section{Rossi's maximum principle\\ and meromorphically convex hulls}
\label{app-Rossi}

The purpose of this appendix is to give another proof of Theorem\/~\rm{\ref{max-rhull}}
using Rossi's general local maximum modulus principle~\cite{Ro}
from the theory of uniform algebras~\cite{Ga}.

Let $\MM(K)\subseteq C(K)$ be the uniform closure
of the algebra of meromorphic functions that are regular on~$K$. 
Its spectrum $S_{\MM(K)}$ is the set of all 
multiplicative functionals $\theta:\MM(K)\to\C$. 
The spectrum is compact in the weak topology on functionals.
There is an isometric embedding $\MM(K)\hookrightarrow C(S_{\MM(K)})$
given by $f\mapsto \hat{f}$ with $\hat{f}(\theta):=\theta(f)$
called the Gelfand transform.

By our definition of the meromorphically convex hull of $K$ in~\S\ref{subsec-lmp-rat-hulls},
a uniformly convergent sequence of meromorphic functions
regular on $K$ converges uniformly on $\mh{K}$.
Hence, every function in $\MM(K)$ has a natural 
continuous extension to $\mh{K}$ with the same norm.
Thus, every point $p\in\mh{K}$ defines 
a functional $\theta_p$ on $\MM(K)$ by taking
the value of the extended function at~$p$.
The map $p\mapsto\theta_p$ is  a continuous map 
from the hull $\mh{K}$ to the spectrum $S_{\MM(S)}$.

The following result is classical for rationally convex hulls in~$\C^n$,
see~\cite[\S III.2]{Ga} or~\cite[\S 1.2]{Sto}.

\begin{prop}
\label{mh-spec}
Let $K$ be a compact subset in a Stein manifold~$M$.
The map $p\mapsto\theta_p$ is a homeomorphism 
$\mh{K}\to S_{\MM(K)}$.
\end{prop}

In particular, the Gelfand transform of $f\in\MM(K)$ is just
the natural extension of $f$ to $\mh{K}$ described above.

\begin{proof}
Fix a holomorphic embedding 
$$
F=(F_1,\ldots,F_n):M\hookrightarrow\C^n.
$$
For a functional $\theta\in S_{\MM(K)}$, set 
$$
z_\theta:=\bigl(\theta(F_1),\ldots,\theta(F_n)\bigr)\in\C^n.
$$ 
Note that $\theta\mapsto z_\theta$ is a continuous map
and 
$$
z_{\theta_p}=\bigl(\theta_p(F_1),\ldots,\theta_p(F_n)\bigr)=F(p)
$$ 
for all $p\in\mh{K}$.

We claim that $z_\theta=F(p_\theta)$ for $p_\theta\in\mh{K}$ 
and that $\theta=\theta_{p_\theta}$, which will prove the proposition.

If $g\in\OO(\C^n)$ with Taylor series expansion $g(z)=\sum\limits_{\alpha\in\Z_+^n} c_\alpha z^\alpha$, then
$$
\begin{array}{rcl}
g(z_\theta) & = &\sum\limits_\alpha c_\alpha \theta(F_1)^{\alpha_1}\cdots \theta(F_n)^{\alpha_n}\\
&= & \theta\Bigl(\sum\limits_\alpha c_\alpha F_1^{\alpha_1}\cdots F_n^{\alpha_n}\Bigr) \\
& =& \theta(g\circ F),
\end{array}
$$
where the second line follows from continuity and multiplicativity.

In particular, if $g$ vanishes on $F(M)$, then $g(z_\theta)=0$.
Since $F(M)$ is the zero set of a family of entire functions by Cartan's Theorem~A, 
we have $z_\theta\in F(M)$, and so $z_\theta=F(p_\theta)$ for some $p_\theta\in M$.

Furthermore, every holomorphic function $f\in\OO(M)$ is the restriction of
an entire function by~\cite[Theorem VIII.A1.8]{GR}, i.e.\ has the form $g\circ F$
for $g\in\OO(\C^n)$. Hence, 
\begin{equation}
\label{theta_on_hol}
f(p_\theta) = \theta(f)
\end{equation}
for all $f\in\OO(M)$.

Suppose now that $p_\theta\notin\mh{K}$ and seek a contradiction.
There exists a section $\tau\in\OO(L)$ of a holomorphic line bundle $L$ on $M$
such that $\tau(p_\theta)=0$ and $\tau$ does not vanish on~$K$.
By Cartan's Theorem~A, there exist a section $\sigma\in\OO(L)$ such that $\sigma(p_\theta)\ne 0$
and a section $\sigma^*\in\OO(L^*)$ of the dual line bundle such that $\sigma^*(p_\theta)\ne 0$.
Therefore, we obtain holomorphic functions $f_1:=\sigma\cdot\sigma^*$ and $f_2:=\tau\cdot\sigma^*$
such that $f_1(p_\theta)\ne 0$, $f_2(p_\theta)=0$, and the meromorphic function
$f:=f_1/f_2 = \sigma/\tau$ is regular on~$K$. We may therefore use~\eqref{theta_on_hol} to write
$$
0\ne f_1(p_\theta) = \theta(f_1) = \theta (f\cdot f_2) 
= \theta(f)\cdot\theta(f_2)= \theta(f)\cdot f_2(p_\theta) =0,
$$
which is the desired contradiction.

Finally, we need to show that
$$
f(p_\theta) = \theta(f)
$$
for all $f\in\MM(K)$ extended to~$\mh{K}$. It is enough to consider 
meromorphic functions $f$ regular on $K$ and hence at $p_\theta\in\mh{K}$. 
Write $f=s/t$ with $s,t\in\OO(L_f)$ and $t$ non-vanishing on~$K$ and at~$p_\theta$. 
Choose again $s^*\in\OO({L_f}^*)$ with $s^*(p_\theta)\ne 0$.
Then $f_1=s\cdot s^*$ and $f_2=t\cdot s^*$ are holomorphic functions on $M$.
Using~\eqref{theta_on_hol}, we obtain
$$
f_1(p_\theta)= \theta(f_1) = \theta (f\cdot f_2) = \theta(f)\cdot\theta(f_2)= \theta(f)\cdot f_2(p_\theta).
$$
Hence,
$$
\theta(f) = \frac{f_1(p_\theta)}{f_2(p_\theta)} = \frac{s(p_\theta)}{t(p_\theta)} = f(p_\theta)
$$
because $f_2(p_\theta)\ne 0$ and $s^*(p_\theta)\ne 0$.
\end{proof}

If $\mh{K}$ is identified with $S_{\MM(K)}$ by
the homeomorphism in Proposition~\ref{mh-spec},
then it follows from Rossi's general 
local maximum principle~\cite[p.~1]{Ro} 
and~\cite[III.8.2]{Ga} that
\begin{equation}
\label{RossiMax}
\max\limits_{\conj{V}} |f| \; \le \; \max\limits_{\conj{V}-V} |f|
\end{equation}
for all meromorphic functions $f$ regular on $K$ and every $V$ open in $\mh{K}-K$. 
Holomorphic functions on $M$ are obviously regular on $K$, 
so~\eqref{RossiMax} implies that
\begin{equation}
\label{local-rhull-phull-A}
 \conj{U} \cap \mh{K} \subseteq \oh{\p U\cap \mh{K}}
\end{equation}
for all open $U\subset M$ such that $U\cap K=\varnothing$.
Thus, we have recovered inclusion~\eqref{local-rhull-phull}.

\begin{proof}[Proof of Theorem\/~\rm{\ref{max-rhull}}]
Suppose that~\eqref{max-rhull-ineq} is false.
By adding to $\phi$ a small multiple of a smooth 
strictly psh function on $M$, we may assume that $\phi$
is strictly psh. Let $p_0\in U\cap\mh{K}$ be a point
at which $\phi$ attains its maximum on $\conj{U}\cap\mh{K}$. 
Since $\phi$ is strictly psh, there is a psh function 
$\wt{\phi}$ on $\conj{U}$ such that $\wt{\phi}(p_0)=\phi(p_0)$
and $\wt{\phi}(q)<\phi(q)$ for all $q\ne p_0$, 
cf.\ Remark~\ref{strict-max}.
Hence, $\wt{\phi}$ attains 
its strict maximum on~$\conj{U}\cap\mh{K}$ 
at~$p_0$.

Let $F=(F_1,...,F_n):M\hookrightarrow\C^n$ be 
a holomorphic embedding such that $F(p_0)=0$.
Take $\eps>0$ so small that  
$$
B_\eps := \left\{q\in M\mid \sum |F_j(q)|^2 < \eps^2\right\} \subset U.
$$
Note that for any compact subset $E\subset \p B_\eps$, 
we have
\begin{equation}
\label{no-hull-on-bdry}
\p B_\eps\cap\oh{E} = E.
\end{equation}
Indeed, for $q_0\in \p B_\eps$, the function 
$q\mapsto \sum F_j(q)\conj{F_j(q_0)}$ 
is holomorphic on $M$ and its real part attains
its strict maximum on $\p B_\eps$ at $q_0$ by Cauchy--Schwarz.

By~\eqref{local-rhull-phull-A} applied to $U=B_\eps$, we have
$$
p_0\in \oh{\p B_\eps \cap\mh{K}},
$$
and by~\eqref{no-hull-on-bdry} with $E=\p B_\eps \cap\mh{K}$, we get
$$
\p B_\eps \cap \oh{\p B_\eps \cap\mh{K}} = \p B_\eps \cap\mh{K}.
$$
Hence, applying Theorem~\ref{max-phull} 
to the compact set $\p B_\eps \cap\mh{K}$ 
and the open set~$B_\eps$, we obtain
$$
\wt{\phi}(p_0)\le \max\limits_{\p B_\eps \cap\mh{K}} \wt{\phi},
$$
which is a contradiction with $\wt{\phi}$ having 
a strict maximum at~$p_0$.
\end{proof}

\begin{rem}
It is worth noting that all known proofs of Rossi's result use methods
from several complex variables~--- 
solutions to Cousin problems and the Oka--Weil theorem.
\end{rem}

\end{document}